\documentclass{svproc}

\usepackage{url}

\usepackage{amsmath}
\usepackage{amssymb}
\usepackage{graphicx}
\usepackage{enumerate}

\newcommand{\C}{\mathbb{C}}
\newcommand{\D}{\mathbb{D}}
\newcommand{\T}{\mathbb{T}}
\DeclareMathOperator{\conv}{conv}
\DeclareMathOperator{\dist}{dist}
\DeclareMathOperator{\sn}{sn}
\DeclareMathOperator{\UCC}{UCC}
\let\Re\undefined
\DeclareMathOperator{\Re}{Re}
\let\Im\undefined
\DeclareMathOperator{\Im}{Im}

\newtheorem{thm}{Theorem}[section]
\numberwithin{thm}{section}
\newtheorem{lem}[thm]{Lemma}
\newtheorem{cor}[thm]{Corollary}

\begin{document}

\mainmatter              

\title{Approximation in the closed unit ball}
\titlerunning{Approximation in the closed unit ball}  

\author{Javad Mashreghi\thanks{Supported by a grant from NSERC} 
\and Thomas Ransford\thanks{Supported by grants from NSERC and the Canada research chairs program.}}
\authorrunning{Javad Mashreghi and Thomas Ransford}

\institute{D\'epartement de math\'ematiques et de statistique, Universit\'e Laval,\\ 
1045 avenue de la M\'edecine, Qu\'ebec (Qu\'ebec), Canada, G1V 0A6\\
\email{javad.mashreghi@mat.ulaval.ca}\\
\email{thomas.ransford@mat.ulaval.ca}}

\maketitle              

\begin{abstract}
In this expository article, we present a number of classic theorems that serve to identify the closure in the sup-norm of  various sets of Blaschke products, inner functions and their quotients, as well as the closure of the convex hulls of these sets. The results presented include theorems of Carath\'eodory, Fisher, Helson--Sarason, Frostman, Adamjan--Arov--Krein, Douglas--Rudin and Marshall. As an application of some of these ideas, we obtain a simple  proof of the Berger--Stampfli spectral mapping theorem for the numerical range of an operator.
\end{abstract}

\section{Introduction}

Let $X$ be a Banach space and let $E$ be a subset of $X$. 
The \emph{convex hull} of $E$, denoted by $\conv(E)$, 
is the set of all elements of the form 
$\lambda_1 x_1 + \lambda_2 x_2 + \cdots + \lambda_n x_n$, 
where $x_j\in E$ and
where $0 \leq \lambda_j \leq 1$ 
with $\lambda_1 + \lambda_2 + \cdots + \lambda_n =1$. 
The \emph{closed unit ball} of $X$ is
\[
\mathbf{B}_{X} = \{ x \in X : \|x\|\leq1 \}.
\]
In this survey, we consider some Banach spaces of functions 
on the open unit disc $\D$ or on the unit circle $\T$, 
e.g., $L^\infty(\T)$, $\mathcal{C}(\T)$, $H^\infty$ and the disc algebra $\mathcal{A}$, 
and explore the norm closure of some subsets of $\mathbf{B}_X$ and of their convex hulls. 

The unimodular elements of the above function spaces enter naturally into our discussion. 
The unimodular elements of $H^\infty$, denoted by $\mathbf{I}$, 
are a celebrated family that are called \emph{inner functions}. 
For other function spaces we use the notation $\mathbf{U}_{X}$ 
to denote the family of unimodular elements of $X$, e.g.,
\[
\mathbf{U}_{\mathcal{C}(\T)} := 
\{ f \in \mathcal{C}(\T) : |f(\zeta)| = 1 \mbox{ for all } \zeta \in \T\}.
\]

The prototype candidates for $E$ are the set of finite Blaschke products ($\mathbf{FBP}$), 
the set of all Blaschke products ($\mathbf{BP}$), 
the inner functions ($\mathbf{I}$), 
the measurable unimodular functions, as well as the quotients of pairs of functions in these families. 
We now summarize the main results. 
The formal statements and attributions will be detailed in the sections that follow.

Finite Blaschke products are elements of the disc algebra $\mathcal{A}$. 
In particular, when considered as functions on $\T$, 
they are elements of $\mathcal{C}(\T)$. 
In this regard, for these elements and their quotients, we shall see that:
\[
\boxed{
\begin{aligned}
\overline{\mathbf{FBP}} &= \mathbf{FBP}\\
\overline{\conv(\mathbf{FBP})} &=  \mathbf{B}_{\mathcal{A}}\\
\overline{\mathbf{FBP}/\mathbf{FBP}} &= \mathbf{U}_{\mathcal{C}(\T)}\\
\overline{\conv(\mathbf{FBP}/\mathbf{FBP})} &=  \mathbf{B}_{\mathcal{C}(\T)}.
\end{aligned}
}
\]

Infinite Blaschke products are elements of the Hardy space $H^\infty$. 
In particular, when considered as functions on $\T$, 
they are elements of $L^\infty(\T)$. 
For these functions and their quotients, we shall see that:
\[
\boxed{
\begin{aligned}
\overline{\mathbf{BP}} &= \overline{\mathbf{I}} = \mathbf{I}\\
\overline{\conv(\mathbf{BP})} &=  \overline{\conv(\mathbf{I})} =  \mathbf{B}_{H^\infty}\\
\overline{\mathbf{BP}/\mathbf{BP}} &=  \overline{\mathbf{I}/\mathbf{I}} =  \mathbf{U}_{L^\infty(\T)}\\
\overline{\conv(\mathbf{BP}/\mathbf{BP})} &= \overline{\conv(\mathbf{I}/\mathbf{I})} =  \mathbf{B}_{L^\infty(\T)}.
\end{aligned}}
\]

In all the results above, we consider the norm topology.
For the Hardy space $H^\infty$, 
and thus \textit{a priori} for the disc algebra $\mathcal{A}$, 
there is a weaker topology which is obtained via semi-norms
\[
p_r(f) := \max_{|z| \leq r} |f(z)|.
\]
This is referred as the topology of uniform convergence on compact subsets (UCC) of $\D$. 
Naively speaking, it is easier to converge under the latter topology.  
Therefore, in some cases we will also study the UCC-closure of a set or its convex hull. 
We shall see that:
\[
\boxed{
\overline{\mathbf{FBP}}^{\UCC} = \mathbf{B}_{H^\infty}.
}
\]
Since on the one hand, 
$\mathbf{FBP}$ is the smallest approximating set in our discussion, 
and on the other hand $\mathbf{B}_{H^\infty}$ is the largest possible set that we can approximate, 
this last result closes the door on any further investigation regarding the UCC-closure.

\section{Approximation on $\D$ by finite Blaschke products} \label{S:app-on-d-fbp}

\begin{center}
\fbox{Goal: $\overline{\mathbf{FBP}} = \mathbf{FBP}$}
\end{center}

Let $f \in H^\infty$, 
and suppose that there is a sequence of finite Blaschke products 
that converges uniformly on $\D$ to $f$. 
Then, by continuity, we also have uniform convergence on $\overline{\D}$. 
Therefore $f$ is necessarily a continuous function on $\overline{\D}$, 
and moreover it is a unimodular function on $\T$. 
It is an easy exercise to show that this function is necessarily a finite Blaschke product. 
A slightly more general version of this result is stated below.

\begin{lem}[Fatou \cite{MR1504825}] \label{T:extmfinb}
Let $f$ be holomorphic in the open unit disc $\D$ and suppose that
\[
\lim_{|z| \to 1} |f(z)| = 1.
\]
Then $f$ is a finite Blaschke product.
\end{lem}

\begin{proof}
Since $f$ is holomorphic on $\D$ and $|f|$  tends uniformly to $1$ as we approach $\T$, 
it has a finite number of zeros in $\D$.
Let $B$ be the finite Blaschke product formed with the zeros of $f$.
Then $f/B$ and $B/f$ are both holomorphic in $\D$, 
and their moduli uniformly tend to $1$ as we approach $\T$. 
Hence, by the maximum principle, 
$|f/B| \leq 1$ and $|B/f| \leq 1$ on $\overline{\D}$.
Thus $f/B$ is constant on $\overline{\D}$, 
and the constant has to be unimodular.\qed
\end{proof}

Lemma~\ref{T:extmfinb} immediately implies the following result.

\begin{thm} \label{T:approximation-D-FBP}
The set $\mathbf{FBP}$ of finite Blaschke products is a closed subset of $\mathbf{B}_{\mathcal{A}}$ 
(and hence also a closed subset of $\mathbf{B}_{H^\infty}$).
\end{thm}

The following result is another simple consequence of Lemma \ref{T:extmfinb}.
It will be needed in later approximation results in this article (see Theorem~\ref{T:helson-sarason}).

\begin{cor} \label{C:extmfinb}
Let $f$ be meromorphic in the open unit disc $\D$ 
and continuous on the closed unit disc $\overline{\D}$
(as a function into the Riemann sphere). 
Suppose that $f$ is unimodular on the unit circle $\T$. 
Then $f$ is the quotient of two finite Blaschke products.
\end{cor}

\begin{proof}
Since $f$ is unimodular on $\T$, meromorphic in $\D$ and continuous on $\overline{\D}$, 
it has a finite number of poles in $\D$. 
Let $B_2$ be the finite Blaschke product with zeros at the poles of $f$. 
Put $B_1:=B_2f$. 
Then $B_1$ satisfies the hypotheses of Lemma \ref{T:extmfinb}, 
and so it is a finite Blaschke product. 
Thus $f=B_1/B_2$.\qed
\end{proof}


\section{Approximation on compact sets by finite Blaschke products}

\begin{center}
\fbox{Goal: $\overline{\mathbf{FBP}}^{\UCC} = \mathbf{B}_{H^\infty}$}
\end{center}

If $f$ is holomorphic on $\D$ 
and can be uniformly approximated on $\D$ by a sequence of finite Blaschke products, 
we saw that, by Lemma \ref{T:extmfinb}, $f$ is itself a finite Blaschke product. 
A general element of $\mathbf{B}_{H^\infty}$ is far from being a finite Blaschke product
and cannot be approached uniformly on $\D$ by finite Blaschke products.
Nevertheless, a weaker type of convergence does hold. 
The following result says that, 
if we equip $H^\infty$ with the topology of uniform convergence on compact subsets of $\D$, 
then the family of finite Blaschke products form a dense subset of $\mathbf{B}_{H^\infty}$. 
In a certain sense, this theorem circumscribes all the other results in this article.

\begin{thm}[Carath\'eodory] \label{T:caratheodory-fbp}
Let $f \in \mathbf{B}_{H^\infty}$. 
Then there is a sequence of finite Blaschke products 
that converges uniformly to $f$ on each compact subset of~$\D$.
\end{thm}

\begin{proof}
(This proof is taken from \cite[page 5]{MR2261424}.)
We construct a finite Blaschke product $B_n$
such that the first $n+1$ Taylor coefficients of $f$ and $B_n$ are equal. 
Then, by Schwarz's lemma, we have
\[
|f(z) - B_n(z)| \leq 2   |z|^n, \hspace{1cm} (z \in \D),
\]
and thus the sequence $(B_n)$ converges uniformly to $f$ on  compact subsets of $\D$.

Let $c_0 := f(0)$. As $f$ lies in the unit ball, $c_0 \in \overline{\D}$. 
If $|c_0|=1$, then, by the maximum principle, 
$f$ is a unimodular constant, and the result is obvious. 
So let us assume that $|c_0| < 1$. Writing
\[
\tau_a(z):=\frac{a-z}{1-\overline{a}z}
\qquad(a,z\in\D),
\]
let us set
\begin{equation} \label{E:tdefb0}
B_0(z):= -\tau_{-c_0}(z) = \frac{z+c_0}{1+\overline{c}_0   z}, \hspace{1cm} (z \in \D).
\end{equation}
Clearly, $B_0$ is a finite Blaschke product and its constant term is $c_0$.

The rest is by induction. 
Suppose that we can construct $B_{n-1}$ for each element of $\mathbf{B}_{H^\infty}$. Set
\begin{equation} \label{E:tdefgbyf}
g(z) := \frac{\tau_{c_0}(  f(z)  )}{z}, \hspace{1cm} (z \in \D).
\end{equation}
By Schwarz's lemma, $g \in \mathbf{B}_{H^\infty}$. 
Hence, there is a finite Blaschke product $B_{n-1}$ such that 
$g-B_{n-1}$ has a zero of of order at least $n$ at the origin. 
If $B$ is a finite Blaschke product of degree $n$, 
and  $w \in \D$, 
then it is easy to verify directly that $\tau_w \circ B$ and $B \circ \tau_w$ 
are also finite Blaschke products of order $n$. Hence
\begin{equation} \label{E:tgnseq}
B_n(z) := \tau_{c_0}(  zB_{n-1}(z)  ), \hspace{1cm} (z \in \D),
\end{equation}
is a finite Blaschke product. Since
\[
f(z) = \tau_{c_0}(  zg(z)  ), \hspace{1cm} (z \in \D),
\]
we naturally expect that $B_n$ does that job. 
To establish this conjecture, it is enough to observe that
\[
\tau_{c_0}(z_2) - \tau_{c_0}(z_1) 
= \frac{(1-|c_0|^2)   (z_1-z_2)}{(1-\overline{c}_0   z_1)   (1-\overline{c}_0   z_2)}.
\]
Hence, thanks to the presence of the factor $z(g(z)-B_{n-1}(z))$, the difference
\[
f(z) - B_n(z) = \tau_{c_0}(  zg(z)  ) - \tau_{c_0}(  zB_{n-1}(z)  )
\]
is divisible by $z^{n+1}$.\qed
\end{proof}

\emph{Remark.}
The equation (\ref{E:tgnseq}) is perhaps a bit misleading, 
as if we have a recursive formula for the sequence $(B_n)_{n \geq 0}$. 
A safer way is to write the formula as
\[
B_{n,f}(z) := \tau_{c_0}(  zB_{n-1,g}(z)  ), \hspace{1cm} (n \geq 1),
\]
where $g$ is related to $f$ via (\ref{E:tdefgbyf}). 
Let us compute an example by finding $B_1:=B_{1,f}$
We know that
\[
B_{1,f}(z) = \tau_{c_0}(  zB_{0,g}(z)  ).
\]
Write $f(z) = c_0+c_1z+\cdots$ and observe that
\[
g(z) = \frac{\tau_{c_0}(  f(z)  )}{z} = \frac{-c_1}{1-|c_0|^2} + O(z).
\]
Then, by (\ref{E:tdefb0}), we have
\[
B_{0,g}(z) = -\tau_{\frac{c_1}{1-|c_0|^2}}(z),
\]
and so we get
\[
B_1(z) = \tau_{c_0}(  -z   \tau_{\frac{c_1}{1-|c_0|^2}}(z)  ).
\]
One may directly verify that
\[
B_1(z) = c_0+c_1z+ O(z^2),
\]
as required.


\section{Approximation on $\D$ by convex combinations of finite Blaschke products}

\begin{center}
\fbox{Goal: $\overline{\conv(\mathbf{FBP})} =  \mathbf{B}_{\mathcal{A}}$}
\end{center}

As we saw in Section \ref{S:app-on-d-fbp}, 
if a function $f \in H^\infty$ can be uniformly approximated 
by a sequence of finite Blaschke products on $\D$, 
then $f$ is continuous on $\overline{\D}$. 
The same result holds 
if we can approximate $f$ 
by elements that are convex combinations of finite Blaschke products. 
The only difference is that, in this case, $f$ is not necessarily unimodular on $\T$. 
We can just say that $\|f\|_\infty \leq 1$. 
More explicitly, the uniform limit of convex combinations of finite Blaschke products 
is a continuous function in the closed unit ball of $H^\infty$. 
It is rather surprising that the converse is also true.

\begin{thm}[Fisher \cite{MR0233995}] \label{T:fisher}
Let $f \in \mathbf{B}_{\mathcal{A}}$, and let $\varepsilon>0$. 
Then there are  finite Blaschke products $B_j$ 
and  convex weights $(\lambda_j)_{1 \leq j \leq n}$ such that
\[
\| \lambda_1 B_1 + \lambda_2 B_2 + \cdots +\lambda_n B_n - f  \|_\infty < \varepsilon.
\]
\end{thm}

\begin{proof}
For $0 \leq t \leq 1$, 
let $f_t(z) := f(tz)$, $z \in \D$. 
Since $f$ is continuous on $\overline{\D}$, we have
\begin{equation} \label{E:tfttofunif}
\lim_{t \to 1} \|f_t-f\|_\infty = 0.
\end{equation}

By Theorem \ref{T:caratheodory-fbp}, 
there is a sequence of finite Blaschke products 
that converges uniformly to $f$ on compact subsets of $\D$. 
Based on our notation, this means that, given $\varepsilon>0$ and $t<1$, 
there is a finite Blaschke product $B$ such that
\[
\|f_t - B_t \|_\infty < \varepsilon/2.
\]
Therefore,  by (\ref{E:tfttofunif}), 
there is a finite Blaschke product $B$ such that
\[
\|f - B_t \|_\infty < \varepsilon.
\]
If we can show that $B_t$ itself is actually a convex combination of finite Blaschke products, 
the proof is done.

Firstly, note that $(gh)_t = g_t   h_t$ 
for all $g$ and $h$, 
and that the family of convex combinations of finite Blaschke products 
is closed under multiplication. 
Hence it is enough only to consider a Blaschke factor
\[
B(z) = \frac{\alpha-z}{1-\overline{\alpha}   z}.
\]
Secondly, it is easy to verify that
\begin{equation} \label{E:tbtcombination}
B_t(z) = \frac{\alpha-tz}{1-\overline{\alpha}   tz} 
= \frac{t(1-|\alpha|^2)}{1-|\alpha|^2   t^2} 
\times \frac{\alpha t -z}{1-\overline{\alpha}t   z} + \frac{|\alpha|   (1-t^2)}{1-|\alpha|^2   t^2} \times e^{i\arg \alpha}.
\end{equation}
The combination on the right side is almost good. 
More precisely, it is a combination of a Blaschke factor and a unimodular constant 
(a special case of a finite Blaschke product), 
with positive coefficients, 
but the coefficients do not add up to one. Indeed, we have
\[
1-\frac{t(1-|\alpha|^2)}{1-|\alpha|^2   t^2} - \frac{|\alpha|   (1-t^2)}{1-|\alpha|^2   t^2} 
= \frac{(1-t)(1-|\alpha|)}{1+|\alpha|t}.
\]
But this obstacle is easy to overcome. We can simply add
\[
0 = \frac{(1-t)(1-|\alpha|)}{2(1+|\alpha|t)} \times 1 + \frac{(1-t)(1-|\alpha|)}{2(1+|\alpha|t)} \times (-1)
\]
to both sides of (\ref{E:tbtcombination}) to obtain 
a convex combination of finite Blaschke products. 
Of course, the factor $1$ in the last identity 
can be replaced by any other finite Blaschke product.\qed
\end{proof}

In technical language, 
Theorem \ref{T:fisher} says that the closed convex hull of finite Blaschke products 
is precisely the closed unit ball of the disc algebra $\mathcal{A}$.


\section{Approximation on $\T$ by quotients of finite Blaschke products}

\begin{center}
\fbox{Goal: $\overline{\mathbf{FBP}/\mathbf{FBP}} = \mathbf{U}_{\mathcal{C}(\T)}$}
\end{center}

If $B_1$ and $B_2$ are finite Blaschke products, 
then $B_1/B_2$ is a continuous unimodular function on $\T$. 
Helson and Sarason showed that 
the family of all such quotients is uniformly dense in the
set of continuous unimodular functions \cite[page 9]{MR0236989}.

To prove the Helson--Sarason theorem, 
we need an auxiliary lemma. 

\begin{lem} \label{L:paritygamma}
Let $f \in \mathbf{U}_{\mathcal{C}(\T)}$. 
Then there exists $g \in \mathbf{U}_{\mathcal{C}(\T)}$ such that either
\[
f(\zeta) = g^2(\zeta)
\]
or
\[
f(\zeta) = \zeta g^2(\zeta)
\]
for all $\zeta \in \T$.
\end{lem}

\begin{proof}
Since $f: \T \longrightarrow \T$ is uniformly continuous, 
we can take $N$ so big that $|\theta-\theta'| \leq 2\pi/N$ implies
$|f(  e^{i\theta}  ) - f(  e^{i\theta'}  )| < 2$.
Now, we divide $\T$ into $N$ arcs
\[
T_k = \left\{   e^{i\theta}   :   \frac{2(k-1)\pi}{N} \leq \theta
\leq \frac{2k\pi}{N}  \right\}, \hspace{1cm} (1 \leq k \leq N).
\]
Then $f(  T_k  )$ is a closed arc in a semicircle, 
and thus there is a continuous function $\phi_k(\theta)$ on the interval
$\frac{2(k-1)\pi}{N} \leq \theta \leq \frac{2k\pi}{N}$ such that
\[
f(e^{i\theta}) = \exp(  i \phi_k(\theta)  ), \hspace{1cm} (e^{i\theta} \in T_k).
\]
These functions are uniquely defined up to  additive multiples of $2\pi$. 
We adjust those additive constants so that
\[
\phi_k( 2k\pi/N ) = \phi_{k+1}( 2k\pi/N )
\]
for $k=1,2,\dots,N-1$.
Define $\phi(\theta) := \phi_k(\theta)$ 
for $\frac{2(k-1)\pi}{N} \leq \theta \leq \frac{2k\pi}{N}$, $k=1,2,\dots,N$. 
Then we get a continuous function
$\phi(\theta)$ on $[0,2\pi]$ such that
\[
f(  e^{i\theta}  ) =
\exp(  i \phi(\theta)  ), \hspace{1cm} (e^{i\theta} \in \T).
\]
Since 
\[ 
\exp(i(\phi(2\pi)-\phi(0)))=f(e^{2\pi i})/f(e^{0i})=1,
\]  
$\phi(2\pi)-\phi(0)$ is an integer multiple of $2\pi$. 
If $(\phi(2\pi)-\phi(0))/2\pi$ is \emph{even}, then set
\[
g(e^{i\theta}) := \exp(  i \phi(\theta)/2  ),
\]
and if $(\phi(2\pi)-\phi(0))/2\pi$ is \emph{odd}, then set
\[
g(e^{i\theta}) := \exp(  i (\phi(\theta)-\theta)/2  ).
\]
Then $g$ is a continuous unimodular function on $\T$ such that
either $f(e^{i\theta}) = g^2(e^{i\theta})$ 
or $f(e^{i\theta}) = e^{i\theta}   g^2(e^{i\theta})$ for all $e^{i\theta} \in \T$.\qed
\end{proof}

\begin{thm}[Helson--Sarason \cite{MR0236989}] \label{T:helson-sarason}
Let $f \in \mathbf{U}_{\mathcal{C}(\T)}$  and let $\varepsilon>0$. 
Then there are finite Blaschke products $B_1$ and $B_2$ such that
\[
\bigg\|  f - \frac{B_1}{B_2}  \bigg\|_{\mathcal{C}(\T)} <\varepsilon.
\]
\end{thm}

\begin{proof}
According to Lemma \ref{L:paritygamma}, 
it is enough to prove the result for unimodular functions of the form $f=g^2$ 
(note that $b(e^{i\theta}): = e^{i\theta}$ is a Blaschke factor). 
Without loss of generality, assume that $\varepsilon < 1$.

By Weierstrass's theorem, there is a trigonometric polynomial $p(z)$ such that
\[
\|g-p\|_{_{\mathcal{C}(\T)}} < \varepsilon.
\]
The restriction $\varepsilon<1$ ensures that $p$ has no zeros on $\T$. 
Let $p^*(z) := \overline{p(1/\overline{z})}$, 
and consider the quotient $p/p^*$. 
Since $p$ is a good approximation to $g$, 
we expect that $p/p^*$ should be a good approximation to $g/g^*=g^2=f$. 
More precisely, on the unit circle $\T$, we have
\[
\frac{g}{g^*} - \frac{p}{p^*} = \frac{(g-p)p^*+(p^*-g^*)p}{g^*   p^*},
\]
which gives
\[
|f-p/p^*| \leq |g-p|+|p^*-g^*| \leq 2 \varepsilon.
\]
It is enough now to note that $p/p^*$ is a meromorphic function 
that is unimodular and continuous on $\T$, 
and thus, according to Corollary~\ref{C:extmfinb}, 
it is the quotient of two finite Blaschke products.\qed
\end{proof}

If we allow approximation by quotients of general Blaschke products, then it turns out that we can approximate a much larger class of functions. This is the subject of the Douglas--Rudin theorem, to be established in Section~\ref{S:DR} below.


\section{Approximation on $\D$ by convex combination of quotients of finite Blaschke products}

\begin{center}
\fbox{Goal: $\overline{\conv(\mathbf{FBP}/\mathbf{FBP})} =  \mathbf{B}_{\mathcal{C}(\T)}$}
\end{center}

The quotient of two finite Blaschke products is a continuous unimodular function on $\T$. Hence a convex combination of such fractions stays in the closed unit ball of $\mathcal{C}(\T)$. As the first step in showing that this set is dense in $\mathbf{B}_{\mathcal{C}(\T)}$, we consider the larger set of all unimodular elements of $\mathcal{C}(\T)$, and then pass to the special subclass of quotients of finite Blaschke products.

\begin{lem} \label{L:cchunimodinl-count}
Let $f \in \mathbf{B}_{\mathcal{C}(\T)}$ and let $\varepsilon>0$. 
Then there are $u_j \in \mathbf{U}_{\mathcal{C}(\T)}$ 
and  convex weights $(\lambda_j)_{1 \leq j \leq n}$ such that
\[
\|  \lambda_1   u_1 + \lambda_2   u_2 + \cdots + \lambda_n
u_n - f  \|_{\mathcal{C}(\T)} < \varepsilon.
\]
\end{lem}

\begin{proof}
Let $w \in \D$. Then, by the Cauchy integral formula,
\[
w = \frac{1}{2\pi i} \int_\T \frac{\zeta+w}{\zeta(1+\overline{w}
\zeta)}  \, d\zeta = \frac{1}{2\pi} \int_{0}^{2\pi}
\frac{e^{i\theta}+w}{1+\overline{w}   e^{i\theta}}  \, d\theta.
\]
In a sense, the integral on the right side 
is an infinite convex combination of unimodular elements. 
We shall approximate it by a Riemann sum and thereby obtain an ordinary finite convex combination. 
Since
\[
\bigg|  \frac{e^{i\theta}+w}{1+\overline{w}   e^{i\theta}} -
\frac{e^{i\theta'}+w}{1+\overline{w}   e^{i\theta'}}  \bigg|
\leq \frac{1+|w|}{1-|w|}    |\theta-\theta'|,
\]
for
\[
\Delta_N := \bigg|  w - \frac{1}{N}   \sum_{k=1}^{N}
\frac{e^{i2k\pi/N}+w}{1+\overline{w}   e^{i2k\pi/N}}
  \bigg| ,
\]
we obtain the estimation
\begin{eqnarray*}
\Delta_N &=&  \bigg|  \frac{1}{2\pi}   \int_{0}^{2\pi}
\frac{e^{i\theta}+w}{1+\overline{w}   e^{i\theta}}  \, d\theta -
\frac{1}{N}   \sum_{k=1}^{N} \frac{e^{i2k\pi/N}+w}{1+\overline{w}
e^{i2k\pi/N}}
  \bigg|\\
&\leq& \frac{1}{2\pi}   \sum_{k=1}^{N}
\int_{2(k-1)\pi/N}^{2k\pi/N} \bigg|
\frac{e^{i\theta}+w}{1+\overline{w}   e^{i\theta}} -
\frac{e^{i2k\pi/N}+w}{1+\overline{w} e^{i2k\pi/N}}
  \bigg|  \, d\theta\\
&\leq& \frac{1}{2\pi}   \sum_{k=1}^{N}
\int_{2(k-1)\pi/N}^{2k\pi/N} \frac{1+|w|}{1-|w|}
\frac{2\pi}{N}  \, d\theta \\
&=& \frac{1+|w|}{1-|w|}
\frac{2\pi}{N}.
\end{eqnarray*}
As $\|f\|_\infty \leq 1$, we have $|f(e^{i\theta})| \leq 1$ for all $e^{i\theta} \in \T$. 
Hence, by the estimate above,
\[
\bigg|  (1-\varepsilon)f(e^{i\theta}) - \frac{1}{N}
\sum_{k=1}^{N}
\frac{e^{i2k\pi/N}+(1-\varepsilon)f(e^{i\theta})}{1+(1-\varepsilon)\overline{f(e^{i\theta})}
  e^{i2k\pi/N}}
  \bigg| \leq \frac{1+|(1-\varepsilon)f(e^{i\theta})|}{1-|(1-\varepsilon)f(e^{i\theta})|}
\frac{2\pi}{N}.
\]
Thus, for each $e^{i\theta} \in \T$,
\[
\bigg|  f(e^{i\theta}) - \frac{1}{N}   \sum_{k=1}^{N}
\frac{e^{i2k\pi/N}+(1-\varepsilon)f(e^{i\theta})}{1+(1-\varepsilon)\overline{f(e^{i\theta})}
  e^{i2k\pi/N}}
  \bigg| \leq
\varepsilon+\frac{4\pi}{\varepsilon N}.
\]
But, each
\[
u_k(e^{i\theta}) =
\frac{e^{i2k\pi/N}+(1-\varepsilon)f(e^{i\theta})}{1+(1-\varepsilon)\overline{f(e^{i\theta})}
  e^{i2k\pi/N}}
\]
is in fact a unimodular continuous function on $\T$. 
Thus, given $\varepsilon>0$, 
it is enough to choose $N$ so large that 
$4\pi/(\varepsilon N) <\varepsilon$, to get
\[
\bigg|  f(e^{i\theta}) - \frac{1}{N}   \sum_{k=1}^{N}
u_k(e^{i\theta})
  \bigg| \leq
2\varepsilon
\]
for all $e^{i\theta} \in \T$.\qed
\end{proof}

In the light of Theorem~\ref{T:helson-sarason}, 
it is now easy to pass from an arbitrary unimodular element 
to the quotient of two finite Blaschke products.

\begin{thm} \label{T:cchui/idinl-count}
Let $f \in \mathbf{B}_{\mathcal{C}(\T)}$ and let $\varepsilon>0$. 
Then there are finite Blaschke products $B_{ij}$, $1 \leq i,j \leq n$ 
and  convex weights $(\lambda_j)_{1 \leq j \leq n}$ such that
\[
\bigg\|  \lambda_1   \frac{B_{11}}{B_{12}} + \lambda_2
\frac{B_{21}}{B_{22}} + \cdots + \lambda_n
\frac{B_{n1}}{B_{n2}} - f  \bigg\|_{\mathcal{C}(\T)} < \varepsilon.
\]
\end{thm}

\begin{proof}
By Lemma \ref{L:cchunimodinl-count}, 
there are $u_j \in \mathbf{U}_{\mathcal{C}(\T)}$ 
and  convex weights $(\lambda_j)_{1 \leq j \leq n}$ such that
\[
\|  \lambda_1   u_1 + \lambda_2   u_2 + \cdots + \lambda_n
u_n - f  \|_{\mathcal{C}(\T)} < \varepsilon/2.
\]
For each $k$, by Theorem \ref{T:helson-sarason}, 
there are finite Blaschke products  $B_{k1}$ and $B_{k2}$ such that
\[
\| u_k - B_{k1}/B_{k2}  \|_\infty < \varepsilon/2.
\]
Hence
\[
\bigg\|  f - \sum_{k=1}^{n} \lambda_k   B_{k1}/B_{k2}
 \bigg\|_\infty \leq \bigg\|  f - \sum_{k=1}^{n} \lambda_k
u_k  \bigg\|_\infty + \sum_{k=1}^{n} \lambda_k \|   u_k -B_{k1}/B_{k2}  \|_\infty < \varepsilon.
\]
This completes the proof.\qed
\end{proof}


\section{Approximation on $\D$ by infinite Blaschke products} \label{S:frostman-shift}

\begin{center}
\fbox{Goal: $\overline{\mathbf{BP}} = \overline{\mathbf{I}} = \mathbf{I}$}
\end{center}

We start to describe our approximation problem 
as in the beginning of Section \ref{S:app-on-d-fbp}. 
But, extra care is needed here 
since we are dealing with infinite Blaschke products 
and they are not continuous on $\overline{\D}$. 
Let $f \in H^\infty$ and assume that there is a sequence 
of infinite Blaschke products that converges uniformly on $\D$ to $f$. 
First of all, we surely have $\|f\|_\infty \leq 1$. 
But, we can say more. 
For each  Blaschke product in the sequence, 
there is an exceptional set of Lebesgue measure zero such that 
on the complement the product has radial limits. 
The union of all these exceptional sets still has Lebesgue measure zero, 
and, at all points outside this union, each infinite Blaschke product has a radial limit. 
Therefore, the function $f$ itself 
must have a radial limit of modulus one almost everywhere. 
In technical language, $f$ is an inner function. 
Hence, in short, if we can uniformly approximate an $f \in H^\infty$ 
by a sequence of infinite Blaschke products, 
then $f$ is necessarily an inner function. 
Frostman showed that the converse is also true.

Let $\phi$ be an inner function for the open unit disc. 
Fix $w \in \D$ and consider $\phi_w = \tau_w \circ \phi$, i.e.,
\[
\phi_w(z) = \frac{w-\phi(z)}{1-\overline{w}   \phi(z)}, \hspace{1cm} (z \in \D).
\]
Since $\tau_w$ is an automorphism of the open unit disc 
and $\phi$ maps $\D$ into itself, 
then clearly so does $\phi_w$, 
i.e. $\phi_w$ is also an element of the closed unit ball of $H^\infty$. 
Moreover, for almost all $e^{i\theta} \in \T$,
\[
\lim_{r \to 1} \phi_w(re^{i\theta}) 
= \frac{w-\phi(e^{i\theta})}{1-\overline{w}   \phi(e^{i\theta})}
= - \overline{\phi(e^{i\theta})}  \frac{w-\phi(e^{i\theta})}{  \overline{w} - \overline{\phi(e^{i\theta})}   } \in \T.
\]
Therefore, for each $w \in \D$, 
the function $\phi_w$ is in fact an inner function. 
What is much less obvious is that $\phi_w$ 
has a good chance of being a Blaschke product. 
More precisely, the exceptional set
\[
\mathcal{E}(\phi) := \{  w \in \D : \phi_w \mbox{ is not a Blaschke product}  \}
\]
is small. 
Frostman showed that the Lebesgue measure of $\mathcal{E}(\phi)$ is zero. 
In fact, there is even a stronger version saying that 
the logarithmic capacity of $\mathcal{E}(\phi)$ is zero. 
But the simpler version with measure is enough for our approximation problem. 
We start with a technical lemma.

\begin{lem} \label{T:tt-meanib}
Let $\phi$ be an inner function in the open unit disc $\D$. Then the limit
\[
\lim_{r \to 1} \int_{0}^{2\pi} \log \big|\phi(re^{i\theta})\big|  \, d\theta
\]
exists. Moreover,  $\phi$ is a Blaschke product if and only if
\[
\lim_{r \to 1} \int_{0}^{2\pi} \log \big|\phi(re^{i\theta})\big|  \, d\theta = 0.
\]
\end{lem}

\begin{proof}
Considering the canonical decomposition $\phi=BS_\sigma$, we have
\[
\log \big|\phi(re^{i\theta})\big| = \log \big|
B(re^{i\theta})  \big| -\frac{1}{2\pi} \int_{\T}
\frac{1-r^2}{1+r^2-2r\cos(\theta-t)}   \, d\sigma(t).
\]
Using Fubini's theorem, we obtain
\begin{equation}\label{E:phibp}
\int_{0}^{2\pi}   \log \big|\phi(re^{i\theta})\big|  \,
d\theta = \int_{0}^{2\pi}   \log \big|B(re^{i\theta})\big|
 \, d\theta - \int_{\T}  \, d\sigma(t).
\end{equation}
Thus the main task is to deal with Blaschke products. 

First of all, we have
\begin{equation}  \label{E:intbisneg}
 \frac{1}{2\pi}   \int_{0}^{2\pi} \log \big|B(re^{i\theta})\big|  \, d\theta \leq 0
\end{equation}
for all $r$ with $0 \leq r <1$. 
Now, without loss of generality, we assume that $B(0) \ne 0$, 
since otherwise we can divide $B$ by   $z^m$, 
where $m$ is the order of the zero of $B$ at the origin, 
and this modification does not change the limit. 
Then, by Jensen's formula,
\[
\log |B(0)| = \sum_{|z_n| < r} \log \bigg( \frac{|z_n|}{r}
\bigg) + \frac{1}{2\pi}   \int_{0}^{2\pi} \log |B(re^{i\theta})|  \, d\theta
\]
for all $r$, $0<r<1$. 
Since $B(0)  = \prod_{n=1}^{\infty} |z_n|$, we thus obtain
\[
\frac{1}{2\pi}   \int_{0}^{2\pi} \log \big|B(re^{i\theta})\big|  \, d\theta =  \sum_{|z_n| < r} \log \bigg(
\frac{r}{|z_n|} \bigg) - \sum_{n=1}^{\infty} \log \bigg(
\frac{1}{|z_n|}\bigg).
\]
Given $\varepsilon > 0$, choose $N$ so large that
\[
\sum_{n=N+1}^{\infty} \log \bigg( \frac{1}{|z_n|}\bigg) <
\varepsilon.
\]
Then, for $r>|z_N|$,
\[
\frac{1}{2\pi}   \int_{0}^{2\pi} \log \big|B(re^{i\theta})\big|  \, d\theta 
\geq \sum_{n=1}^{N} \log \bigg(
\frac{r}{|z_n|} \bigg) -  \sum_{n=1}^{N} \log \bigg(
\frac{1}{|z_n|}\bigg) - \varepsilon.
\]
Therefore,
\[
\liminf_{r \to 1} \frac{1}{2\pi}   \int_{0}^{2\pi} \log \big|B(re^{i\theta})\big|  \, d\theta \geq - \varepsilon,
\]
and, since $\varepsilon$ is an arbitrary positive number,
\[
\liminf_{r \to 1} \frac{1}{2\pi}   \int_{0}^{2\pi} \log \big|B(re^{i\theta})\big|  \, d\theta \geq 0.
\]
Finally, (\ref{E:intbisneg}) and the last inequality together imply that
\[
\lim_{r\to1} \frac{1}{2\pi}   \int_{0}^{2\pi} \log \big|B(re^{i\theta})\big|  \, d\theta = 0.
\]
 
Returning now to \eqref{E:phibp}, we see that 
\[
\lim_{r \to 1} \int_{0}^{2\pi} \log \big|\phi(re^{i\theta})\big|  \, d\theta 
= \lim_{r \to 1^-} \int_{0}^{2\pi} \log \big|B(re^{i\theta})\big|  \, d\theta - \int_{\T}  \,d\sigma(t) 
= - \sigma(\T).
\]
This formula also shows that
\[
\lim_{r \to 1} \int_{0}^{2\pi} \log \big|\phi(re^{i\theta})\big|  \, d\theta = 0,
\]
if and only if  $\sigma(\T)=0$, and, 
since $\sigma$ is a positive measure, 
this holds if and only if $\sigma \equiv 0$.  
Therefore, the above limit is zero if and only if $\phi$ is a Blaschke product.\qed
\end{proof}

In view of the following result, 
the functions $\phi_w$, $w \in \D$, are called the \emph{Frostman shifts} of $\phi$.

\begin{lem}[Frostman \cite{MR0012127}] \label{T:frostman-shifts-1}
Let $\phi$ be an inner function for the open unit disc. Fix $0 < \rho < 1$, and define
\[
\mathcal{E}_\rho(\phi) 
:= \{  e^{i\theta} \in \T : \phi_{\rho e^{i\theta}} \mbox{ is not a Blaschke product}  \}.
\]
Then $\mathcal{E}_\rho(\phi)$ has Lebesgue measure zero.
\end{lem}

Note that this theorem implies that 
the two-dimensional Lebesgue measure of $\mathcal{E}(\phi)$ is also zero.

\begin{proof}
For each $\alpha \in \D$, we have
\begin{equation} \label{E:tintlogtphi}
\frac{1}{2\pi}    \int_{0}^{2\pi}   \log \bigg|
\frac{\rho e^{i\theta}-\alpha}{1-\rho e^{-i\theta}\alpha} \bigg|
 \, d\theta = \max(\log \rho, \log|\alpha|).
\end{equation}
Since $\phi$ is inner, 
we can replace $\alpha$ by $\phi(re^{it})$ and then integrate with respect to~$t$.
This gives
\[
\frac{1}{2\pi} \int_{0}^{2\pi} \int_{0}^{2\pi}
\log \left| \frac{\rho e^{i\theta} - \phi(re^{it})}{1-\rho e^{-i\theta}\phi(re^{it})} \right| \, d\theta  dt 
=\int_{0}^{2\pi} \max(\log \rho, \log|\phi(re^{it})|) \, dt.
\]
Since $\rho$ is fixed and $|\phi| \leq 1$, the family
\[
f_r (e^{it}) = \max(\log \rho,
\log|\phi(re^{it})|), \hspace{1cm} (e^{it} \in \T),
\]
where the parameter $r$ runs through $[0,1)$, 
is \emph{uniformly} bounded in modulus by the positive constant $-\log\rho$, and
\[
\lim_{r \to 1} f_r (e^{it}) = \max(\log \rho, \lim_{r \to 1}
\log|\phi(re^{it})|) = \max(\log \rho, 0)=0
\]
for almost all $e^{it} \in \T$. 
Hence, by the dominated convergence theorem,
\[
\lim_{r \to 1}  \int_{0}^{2\pi} \, f_r (e^{it}) \, dt = 0,
\]
which we rewrite as
\[
\lim_{r \to 1}  \int_{0}^{2\pi} \bigg(  \int_{0}^{2\pi}   \log \left| \frac{\rho e^{i\theta} - \phi(re^{it})}{1-\rho e^{-i\theta}\phi(re^{it})} \right|  \, d\theta \bigg) \, dt = 0.
\]
But, considering the fact that the integrand is negative, 
the Fubini theorem gives
\[
\lim_{r \to 1} \int_{0}^{2\pi} \bigg( \int_{0}^{2\pi} \log \left| \frac{\rho e^{i\theta} - \phi(re^{it})}{1-\rho e^{-i\theta}\phi(re^{it})} \right| \, dt  \bigg) \, d\theta = 0.
\]
Set
\[
M(r,\theta) := \int_{0}^{2\pi} -\log \left| \frac{\rho e^{i\theta} - \phi(re^{it})}{1-\rho e^{-i\theta}\phi(re^{it})} \right| \, dt.
\]
Then $M(r,\theta) \geq 0$ for all $r,\theta$, and
\begin{equation} \label{E:frosmean0-1}
\lim_{r \to 1} \int_{0}^{2\pi} M(r,\theta) \, d\theta = 0.
\end{equation}
Now, we put together two facts. 
First, according to Lemma \ref{T:tt-meanib}, 
we know that, for each  $\theta$,
\[
\lim_{r \to 1} M(r,\theta)
\]
exists. Second, by Fatou's lemma,
\[
\int_{0}^{2\pi}\bigg(  \liminf_{r \to 1} M(r,\theta)  \bigg) \, d\theta
\leq
\liminf_{r \to 1} \int_{0}^{2\pi} M(r,\theta) \, d\theta.
\]
Hence, by (\ref{E:frosmean0-1}) and the fact that $M(r,\theta) \geq 0$, we conclude that
\[
\int_{0}^{2\pi}\bigg(  \lim_{r \to 1} M(r,\theta)  \bigg) \, d\theta = 0.
\]
In particular, we must have $\lim_{r \to 1} M(r,\theta) = 0$ for almost all $\theta \in [0,2\pi]$, i.e.,
\[
\lim_{r \to 1} \int_{0}^{2\pi}   \log \big|
\phi_{\rho e^{i\theta}}(re^{it}) \big|  \, dt = 0
\]
for almost all $\theta \in [0,2\pi]$. 
Therefore, again by Lemma \ref{T:tt-meanib}, 
$\phi_{\rho e^{i\theta}}$ is indeed a Blaschke product 
for almost all $\theta \in [0,2\pi]$. 
In other words, $\mathcal{E}_\rho(\phi)$ has Lebesgue measure zero.\qed
\end{proof}

The preceding result immediately implies 
the approximation theorem that we are seeking. 
It shows that the set $\mathbf{BP}$ of Blaschke products is uniformly dense 
in the set of all inner functions $\mathbf{I}$.

\begin{thm}[Frostman \cite{MR0012127}] \label{T:frostman-approx}
Let $\phi$ be an inner function in the open unit disc. 
Then, given $\varepsilon>0$, there is a Blaschke product $B$ such that
\[
\|\phi-B\|_\infty < \varepsilon.
\]
\end{thm}

\begin{proof}
Take $\rho \in (0,1)$ small enough so that $2\rho/(1-\rho) < \varepsilon$. 
According to Lemma \ref{T:frostman-shifts-1}, 
on the circle $\{|z|=\rho\}$ 
there are many candidates $\rho e^{i\theta}$ such that 
$\phi_{\rho e^{i\theta}}$ is a Blaschke product. 
Pick any one of them. Then, we have
\[
|\phi(z)+\phi_{\rho e^{i\theta}}(z)| = \bigg|  \frac{\rho
e^{i\theta}-\rho e^{-i\theta}\phi^2(z)}{1-\rho e^{-i\theta}\phi(z)}
 \bigg| \leq \frac{2\rho}{1-\rho} < \varepsilon
\]
for all $z \in \D$. This simply means that
$\|\phi+\phi_{\rho e^{i\theta}}\|_\infty < \varepsilon$.
Now take $B:=-\phi_{\rho e^{i\theta}}$.\qed
\end{proof}

Frostman's approximation result (Theorem \ref{T:frostman-approx}) 
should be compared with Carath\'eodory theorem (Theorem \ref{T:caratheodory-fbp}). 
On one hand, the approximation in Frostman's result is stronger. 
The convergence is uniform on $\D$, 
and not just on a fixed compact subset of $\D$. 
But, on the other hand, it only applies 
to a smaller class of functions (inner functions) 
in the closed unit disc of $H^\infty$.

Theorem \ref{T:frostman-approx} may also be considered 
as a generalization of Theorem \ref{T:approximation-D-FBP}. 
In the latter, we consider a small set of Blaschke products (just finite Blaschke products) 
and thus we are not able to approximate all inner functions. 
But Frostman says that, if we enlarge our set 
and consider all Blaschke products, 
then we can approximate all inner functions. 

However, though this interpretation is true, it is not the whole truth. 
Theorem \ref{T:approximation-D-FBP} says that 
the set of finite Blaschke products is a closed subset of $\partial\mathbf{B}_{H^\infty}$. 
Then Theorem \ref{T:frostman-approx} says that 
its complement in the family of inner functions, i.e., $\mathbf{I} \setminus \mathbf{FBP}$, 
is also a closed subset of $\partial\mathbf{B}_{H^\infty}$ 
such that infinite Blaschke products are uniformly dense in $\mathbf{I} \setminus \mathbf{FBP}$. 
In fact, by considering zeros, it is easy to see that
\[
\mbox{dist}_{\infty} (\mathbf{FBP},~\mathbf{I} \setminus \mathbf{FBP}) \geq 1,
\]
i.e., both parts are well separated on the boundary of $\mathbf{B}_{H^\infty}$.


\section{Existence of unimodular functions in the coset $f + H^\infty(\T)$}

To study duality on Hardy spaces, 
we recall some well-known facts from functional analysis. 
Let $X$ be a Banach space, and let $X^*$ denote its dual space. 
Let $A$ be a closed subspace of $X$. The \emph{annihilator} of $A$ is
\[
A^\perp := \{   \Lambda \in X^* :   \, \Lambda(a)
= 0 \text{~for all~} a \in A \},
\]
which is a closed subspace of $X^*$. 
The canonical projection of $X$ onto the quotient space $X/A$ is defined by
\[
\begin{array}{cccc}
\pi: & X & \longrightarrow & X/A\\
& x & \longmapsto & x+A.
\end{array}
\]
For each $x \in X$, by the definition of norm in the quotient space $X/A$,  we  have
\begin{equation} \label{D:distaA}
\dist(x,A) = \inf_{a \in A} \|x-a\| = \| \pi (x) \|_{_{X/A}}.
\end{equation}
Using the Hahn--Banach theorem from functional analysis, we have
\[
\dist(x,A) = \sup_{\Lambda \in A^\perp,
\|\Lambda\|_{_{X^*}}=1}    |\Lambda(x)|.
\]
Moreover, the supremum is attained, i.e., there is $\Lambda_0 \in
A^\perp$ with $\|\Lambda_0\|_{_{X^*}}=1$ such that
\[
\dist(x,A) = \Lambda_0(x).
\]
Thanks to these remarks, we obtain the dual identifications 
\[
(X/A)^* = A^\perp
\quad\text{and}\quad
A^* = X^*/A^\perp.
\]

For a Banach space  $X$ of functions defined on the unit circle $\T$, 
we define $X_0$ to be the family of all functions $e^{i\theta} f(e^{i\theta})$
such that $f \in X$. 
In all cases that we consider below, 
$X_0$ is a closed subspace of $X$. 
If $f \in X$ has a holomorphic extension to the open unit disc, 
then the holomorphic extension of $e^{i\theta} f(e^{i\theta})$ would be $zf(z)$, 
a function having a zero at the origin. 
This fact explains the notation $X_0$.

The following lemma summarizes a number of dual identifications
 of interest to us.

\begin{lem}[\protect{\cite[\S IV.1]{MR2261424}}] \label{L:dualoflp/hp}
Let $1 \leq p<\infty$, and let $1/p+1/q=1$. Then:
\begin{enumerate}[\rm(a)]
\item $(  L^p/H^p  )^* = H_0^q$,
\item $(  L^p/H^p_0  )^* = H^q$,
\item $(  H^p  )^* = L^q/H^q_0$,
\item $(  H^p_0  )^* = L^q/H^q$.
\end{enumerate}
\end{lem}

We can apply this method to study 
$\dist(f,H^p(\T)  )$, 
where $f$ is an element of $L^p(\T)$ and $1 \leq p \leq \infty$. 
In the following, we just need the case $p=\infty$.

\begin{thm} \label{T:aooxl1inftyhinfty}
Let $f \in L^\infty(\T)$.  Then the following hold.
\begin{enumerate}[\rm(a)]
\item There exists $g \in H^\infty(\T)$ such that
\[
\dist(  f,H^\infty(\T)  ) = \|f -g\|_\infty.
\]
\item We have
\[
\dist(  f,H^\infty(\T)  ) = \sup_{h \in H^1_0(\T),
\|h\|_1=1} \bigg|
\frac{1}{2\pi}   \int_{0}^{2\pi} f(e^{i\theta})    h(e^{i\theta})   \, d\theta  \bigg|.
\]
\end{enumerate}
\end{thm}

\begin{proof}
(a) By (\ref{D:distaA}), there are $g_n \in H^\infty(\T)$,
$n \geq 1$, such that
\[
\dist(  f,H^\infty(\T)  ) = \lim_{n \to \infty}
\|f-g_n\|_\infty.
\]
Hence,  $(\|g_n\|_1)_{n \geq 1}$ is a bounded sequence in $H^\infty(\T)$. 
By Lemma \ref{L:dualoflp/hp}(b), 
$H^\infty(\T)$ is the dual of $L^1(\T)/H^1_0(\T)$. 
Hence, looking at the sequence $(g_n)_{n \geq 1}$ 
as a family of uniformly bounded linear functionals on
$L^1(\T)/H^1_0(\T)$, 
by the Banach--Alaoglu theorem, 
we can extract a subsequence that is convergent in the
weak* topology of $H^\infty(\T)$. 
More explicitly, 
there exists $g \in H^\infty(\T)$ and a
subsequence $(n_k)_{k \geq 1}$ such that
\[
\lim_{k \to \infty} \frac{1}{2\pi}   \int_{0}^{2\pi} h( e^{i\theta} )    g_{n_k}(
e^{i\theta} )  \, d\theta = \frac{1}{2\pi}   \int_{0}^{2\pi} h(
e^{i\theta} )    g( e^{i\theta} )  \, d\theta
\]
for all $h \in L^1(\T)$. By H\"older's inequality,
\[
\bigg|  \frac{1}{2\pi}   \int_{0}^{2\pi} h( e^{i\theta} )    (  f( e^{i\theta}
)-g_{n_k}( e^{i\theta} )  )  \, d\theta  \bigg|
\leq \|h\|_1    \| f-g_{n_k} \|_\infty.
\]
Let $k \longrightarrow \infty$ to get
\[
\bigg|  \frac{1}{2\pi}   \int_{0}^{2\pi} h( e^{i\theta} )    (  f( e^{i\theta}
)-g( e^{i\theta} )  )  \, d\theta  \bigg| \leq
\|h\|_1    \dist(  f,H^\infty(\T)  )
\]
for all $h \in L^1(\T)$. 

If  $L^1(\T)$ were the dual of $L^\infty(\T)$, 
then we would have been able to use duality techniques 
and the reasoning would have been easier. 
But, since $L^1(\T)$ is a proper subclass of the dual of $L^\infty(\T)$,
we have to proceed differently.

If $E$ is a measurable subset of $\T$, 
then its characteristic function $h = \chi_{_{E}}$ is integrable. 
Hence, with this choice, we obtain
\[
\bigg|  \frac{1}{|E|}    \int_{E} (  f( e^{i\theta} )-g(
e^{i\theta} )  )   \, d\theta  \bigg| \leq
\dist(  f,H^\infty(\T)  )
\]
for all measurable sets $E \subset \T$ 
with $|E| = \int_{E} d\theta \ne 0$. 
This is enough to conclude $\| f-g \|_\infty \leq \dist(f,H^\infty(\T)  )$. 
Note that the reverse inequality $\dist(  f,H^\infty(\T)  ) \leq \| f-g\|_\infty$ 
is a direct consequence of the definition of $\dist(  f,H^\infty(\T)  )$.

(b) By definition,
\[
\dist(  f,H^\infty(\T)  ) =
\|f+H^\infty(\T)\|_{L^\infty(\T)/H^\infty(\T)},
\]
and, by Lemma \ref{L:dualoflp/hp}($d$), we have
$L^\infty(\T)/H^\infty(\T)= (  H^1_0(\T)  )^*$. 
Hence
\begin{eqnarray*}
\dist(  f,H^\infty(\T)  ) &=& \|f+H^\infty(\T)\|_{(H^1_0(\T)  )^*} \\
&=& \sup_{h \in H^1_0(\T),   \|h\|_1=1}    \bigg|
\frac{1}{2\pi}   \int_{0}^{2\pi} f( e^{i\theta} )    h( e^{i\theta} )  \, d\theta  \bigg|.
\end{eqnarray*}
This completes the proof.\qed
\end{proof}

Let $f \in L^\infty(\T)$. 
If the coset $f + H^\infty(\T)$ contains a unimodular element, 
then necessarily $\dist( f,H^\infty(\T)  ) \leq 1$. 
A profound result of Adamjan--Arov--Krein says that,
under the slightly more restrictive condition $\dist( f,H^\infty(\T)  )<1$, 
the reverse implication holds. 
In this section, we discuss this result, 
which will be needed in studying the closed convex hull of Blaschke products. 
We start with a technical lemma.

\begin{lem} \label{L:fninh1cgt0one}
Let $f_n \in H^1(\D)$ with $\|f_n\|_1 \leq 1$, $n \geq 1$.
Suppose that there is a measurable subset $E$ of $\T$ with $|E| \ne 0$ such that
\[
\lim_{n \to \infty} \int_{E} |f_n( e^{i\theta} )|  \, d\theta = 0.
\]
Then
\[
\lim_{n \to \infty}  f_n( 0 )  = 0.
\]
\end{lem}

\begin{proof}
If $|E| = 2\pi$, then the result is an immediate consequence of
the identity
\[
f_n( 0 ) = \frac{1}{2\pi}   \int_{\T} f_n( e^{i\theta} )  \, d\theta
= \frac{1}{2\pi}   \int_{E} f_n( e^{i\theta} )  \, d\theta.
\]
If $0 < |E| < 2\pi$, then, on the one hand,
\[
\frac{1}{|E|} \int_{E} \log |f_n( e^{i\theta} )|  \, d\theta \leq
\log\bigg(  \frac{1}{|E|} \int_{E} |f_n( e^{i\theta} )|  \,
d\theta  \bigg) \longrightarrow -\infty,
\]
as $n \longrightarrow \infty$, and, on the other hand,
\begin{eqnarray*}
\frac{1}{|\T \setminus E|} \int_{\T \setminus E} \log |f_n(
e^{i\theta} )|  \, d\theta &\leq& \log\bigg(  \frac{1}{|\T
\setminus E|} \int_{\T \setminus E} |f_n( e^{i\theta} )|  \,
d\theta  \bigg) \\
&\leq& \log\bigg(  \frac{ \|f_n\|_1}{|\T \setminus E|}  \bigg)
\leq \log\bigg(  \frac{1}{|\T \setminus E|}  \bigg).
\end{eqnarray*}
Therefore,
\[
\lim_{n \to \infty}  \int_{0}^{2\pi} \log|f_n( e^{i\theta} )|  \,
d\theta  = -\infty.
\]
Finally, since $\log|f_n|$ is a subharmonic function, we have
\[
|f_n(0)| \leq \exp\bigg(   \frac{1}{2\pi}   \int_{0}^{2\pi} \log|f_n( e^{i\theta})|  \, d\theta   \bigg),
\]
and thus $f_n(0) \longrightarrow 0$ as $n \longrightarrow \infty$.\qed
\end{proof}

The space $H^\infty(\T)$ contains all inner functions, which are elements of modulus one. 
The following result shows that if we slightly perturb $H^\infty(\T)$ in $L^\infty(\T)$, 
it still contains unimodular elements.

\begin{thm}[Adamjan--Arov--Krein \cite{MR0636333,MR0234274,MR0298454}] \label{T:adamarovkrien}
Let $f \in L^\infty(\T)$ be such that
\[
\dist( f,H^\infty(\T)  )<1.
\]
Then there exists an $\omega \in f + H^\infty(\T)$ with
\[
|\omega(e^{i\theta})| = 1
\]
for almost all $e^{i\theta} \in \T$.
\end{thm}

\begin{proof}
(Garnett \cite[page 150]{MR1669574}) The proof is long and thus we divide it into several steps.

\medbreak
\noindent\emph{Step 1:} Definition of $\omega$ as the solution of an extremal problem.

\medskip
Since $\dist(  f,H^\infty(\T) )<1$, the set
\[
\mathcal{E} := \{  \omega  :  \omega \in f + H^\infty(\T),
\|\omega\|_\infty \leq 1  \}
\]
is not empty. Let
\[
\alpha 
:= \sup_{\omega \in \mathcal{E}}  \bigg|  \frac{1}{2\pi}   \int_{0}^{2\pi} \omega(e^{i\theta} )  \, d\theta  \bigg|.
\]
We show that the supremum is attained. 
There are $\omega_n = f+g_n \in \mathcal{E}$, $n \geq 1$, such that
\[
\lim_{n \to \infty} \bigg|  \frac{1}{2\pi}   \int_{0}^{2\pi} \omega_n(e^{i\theta} )  \, d\theta  \bigg| = \alpha.
\]
Since $g_n \in H^\infty(\T)$ 
and $\|g_n\|_\infty \leq 1 +\|f\|_\infty$, 
there exist $g \in H^\infty(\T)$ 
and a subsequence $(n_k)_{k \geq 1}$ such that
\begin{equation} \label{E:intg-npartomwgan}
\lim_{k \to \infty} \frac{1}{2\pi}   \int_{0}^{2\pi} h( e^{i\theta} )    g_{n_k}( e^{i\theta} )  \, d\theta 
= \frac{1}{2\pi}   \int_{0}^{2\pi} h(e^{i\theta} )    g( e^{i\theta} )  \, d\theta
\end{equation}
for all $h \in L^1(\T)$. 
Indeed, 
since the $g_n$ are uniformly bounded,
we can say that a subsequence $g_{n_k}$ converges weak*
to an element $g \in L^\infty(\T)$. 
But the weak*-convergence implies $\hat{g}(n) = 0$, $n \leq -1$,
so in fact we have $g \in H^\infty(\T)$.

Set
$\omega := f + g$. By (\ref{E:intg-npartomwgan}),
\[
\lim_{k \to \infty} \frac{1}{2\pi}   \int_{0}^{2\pi} h( e^{i\theta} )
\omega_{n_k}( e^{i\theta} )  \, d\theta =
\frac{1}{2\pi}   \int_{0}^{2\pi} h( e^{i\theta} )    \omega( e^{i\theta} )  \, d\theta
\]
for all $h \in L^1(\T)$. This fact implies
\[
\|\omega\|_\infty \leq \liminf_{k \to \infty} \|\omega_{n_k}\|_\infty \leq 1,
\]
which ensures that $\omega \in \mathcal{E}$. Moreover, taking $h \equiv 1$, we get
\[
\bigg|  \frac{1}{2\pi}   \int_{0}^{2\pi} \omega( e^{i\theta} )  \, d\theta  \bigg| = \alpha,
\]
and thus we can write
\begin{equation} \label{E:Tdefteta0}
\frac{1}{2\pi}   \int_{0}^{2\pi} \omega( e^{i\theta} )  \, d\theta = \alpha   e^{i\theta_0}.
\end{equation}

\medbreak
\noindent\emph{Step 2:} $\| \omega \|_\infty = 1$.

\medskip
Let $\omega_1:= \omega +(1-\| \omega \|_\infty)  e^{i\theta_0}$. 
Thus $\omega_1 \in \mathcal{E}$ and, by the definition of $\alpha$,
\[
\bigg|  \frac{1}{2\pi}   \int_{0}^{2\pi}  \omega_1( e^{i\theta} )  \,
d\theta  \bigg| \leq \alpha.
\]
But, by (\ref{E:Tdefteta0}),
\[
\bigg|  \frac{1}{2\pi}   \int_{0}^{2\pi}  \omega_1( e^{i\theta} )  \,
d\theta  \bigg| = |\alpha   e^{i\theta_0}+ (1-\|
\omega \|_\infty)  e^{i\theta_0}| = \alpha + 1-\| \omega
\|_\infty.
\]
Hence $\| \omega \|_\infty \geq 1$. 
We already know that $\| \omega \|_\infty \leq 1$,
and thus $\| \omega \|_\infty = 1$.

\medbreak 

\noindent\emph{Step 3:} $\dist(  \omega,H^\infty_0(\T)  )=1$.

\medskip
Let
$\varepsilon>0$, let $g \in H^\infty_0(\T)$, 
and set 
$\omega_1: = \omega - g + \varepsilon  e^{i\theta_0}$. 
Then $\omega_1 \in f + H^\infty(\T)$, 
and, by (\ref{E:Tdefteta0}),
\[
\bigg|  \frac{1}{2\pi}   \int_{0}^{2\pi}  \omega_1( e^{i\theta} )  \,
d\theta  \bigg| = |\alpha   e^{i\theta_0}+
\varepsilon  e^{i\theta_0}| = \alpha + \varepsilon > \alpha.
\]
Thus, according to the definition of $\alpha$, 
we have $\omega_1 \not\in \mathcal{E}$. Thus
\[
\|  \omega - g + \varepsilon  e^{i\theta_0}  \|_\infty > 1
\]
for all $\varepsilon>0$ and all $g \in H^\infty_0(\T)$. 
Let $\varepsilon \rightarrow 0$ to get
\[
\|  \omega - g   \|_\infty \geq 1
\]
for all $g \in H^\infty_0(\T)$. 
However, when $g \equiv 0$,
we also know that $\|  \omega   \|_\infty = 1$. 
Hence $\dist(  \omega,H^\infty_0(\T)  )=1$.

\medskip
Before moving on to Step~4, we remark that
Theorem \ref{T:aooxl1inftyhinfty}$(b)$,
applied to the function $e^{-i\theta}\omega(e^{i\theta})$,
 implies that there are
$h_n \in H^1(\T)$, $n \geq 1$, with $\|h_n\|_1=1$, such that
\begin{equation} \label{E:h-nomea-nintto1}
\dist(
\omega,H^\infty_0(\T)  )=\lim_{n \to \infty} \frac{1}{2\pi}   \int_{0}^{2\pi} \omega( e^{i\theta} )
h_n( e^{i\theta} )  \, d\theta = 1.
\end{equation}
The extension of $h_n$ to the open unit disc is also denoted by $h_n$.

\medbreak

\noindent\emph{Step 4:} For all measurable sets $E \subset \T$ with $|E| \ne 0$, we have
\[
\liminf_{n \to \infty} \int_{E} |h_n( e^{i\theta})|  \, d\theta> 0.
\]

\medskip
Since $h_n - h_n(0) \in H^1_0(\T)$, by (the easy part of) Theorem~\ref{T:aooxl1inftyhinfty}$(b)$,
\[
\|h_n - h_n(0)\|_1    \dist(  \omega,H^\infty(\T)  )
\geq \bigg|  \frac{1}{2\pi}   \int_{0}^{2\pi} \omega( e^{i\theta} )  
(h_n(e^{i\theta} ) - h_n(0)  )  \, d\theta  \bigg|,
\]
and, by (\ref{E:Tdefteta0}),
\[
\frac{1}{2\pi}   \int_{0}^{2\pi} \omega( e^{i\theta} )  (  h_n( e^{i\theta} )- h_n(0)  )  \, d\theta 
= \frac{1}{2\pi}   \int_{0}^{2\pi} \omega(e^{i\theta} )     h_n( e^{i\theta} )   \, d\theta- h_n(0)   
\alpha   e^{i\theta_0}.
\]
Thus, we have
\[
(1+|h_n(0)|) \dist(  \omega,H^\infty(\T)  ) 
\geq \bigg|  \frac{1}{2\pi} \int_{0}^{2\pi} \omega(e^{i\theta} )h_n( e^{i\theta} ) \, d\theta \bigg| 
- \alpha   |h_n(0)|,
\]
which yields
\[
|h_n(0)| \geq \frac{\bigg|  \frac{1}{2\pi} \int_{0}^{2\pi} \omega(e^{i\theta} ) h_n( e^{i\theta} )  \, d\theta
 \bigg| - \dist(  \omega,H^\infty(\T)  )}{\alpha 
 + \dist(  \omega,H^\infty(\T)  )}.
\]
Let $n \longrightarrow \infty$ to get, by (\ref{E:h-nomea-nintto1}),
\[
\liminf_{n \to \infty}|h_n(0)| \geq \frac{1-\dist(
\omega,H^\infty(\T)  )}{\alpha+\dist(
\omega,H^\infty(\T)  )}.
\]
But $\dist(  \omega,H^\infty(\T)  )
= \dist(  f,H^\infty(\T)  ) < 1$. 
Hence, $\liminf_{n \to\infty} |h_n(0)| > 0$. 
Now, apply Lemma \ref{L:fninh1cgt0one}.

\medbreak

\noindent\emph{Step 5:}  $\omega$ is unimodular.

\medskip
We know that
$|\omega(e^{i\theta})| \leq 1$ for almost all $e^{i\theta} \in \T$. 
Let $0\leq \lambda<1$ and set
\[
E_\lambda:= \{  e^{i\theta}
\in \T  :  |\omega(e^{i\theta})| \leq \lambda  \}.
\]
Then
\[
\bigg|  \int_{0}^{2\pi} \omega( e^{i\theta} ) h_n(e^{i\theta} ) d\theta  \bigg| 
\leq \lambda   \int_{E_\lambda} |h_n( e^{i\theta} )|  \, d\theta +
 \int_{\T \setminus E_\lambda}  |h_n(e^{i\theta} )|   \, d\theta.
\]
Since $\|h_n\|_1 = 1$,
\[
\frac{1}{2\pi}    \int_{\T \setminus E_\lambda}  |h_n(e^{i\theta} )|   \, d\theta 
= 1 - \frac{1}{2\pi}    \int_{E_\lambda}  |h_n(e^{i\theta} )|   \, d\theta,
\]
and thus
\[
\frac{1-\lambda}{2\pi}   \int_{E_\lambda} |h_n( e^{i\theta} )|  \,d\theta 
\leq 1 - \bigg|  \frac{1}{2\pi}  \int_{0}^{2\pi} \omega(e^{i\theta} )    h_n( e^{i\theta} )   \, d\theta \bigg|.
\]
By (\ref{E:h-nomea-nintto1}), this inequality implies that
\[
\lim_{n \to \infty} \int_{E_\lambda} |h_n( e^{i\theta} )|  \,d\theta = 0.
\]
Therefore, by Step 4, $|E_\lambda| = 0$ for all $0\leq \lambda<1$.\qed
\end{proof}

Theorem
\ref{T:adamarovkrien} has a geometric interpretation. 
Let $\mathcal{U}(\T)$ denote 
the family of all unimodular functions in $L^\infty(\T)$. 
Then Theorem~\ref{T:adamarovkrien} says that 
the open unit ball of $L^\infty(\T)$
is a subset of $H^\infty(\T) + \mathcal{U}(\T)$.

\begin{cor} \label{C:arovadamkrein}
Let $f \in H^\infty(\T)$ with $\|f\|_\infty < 1$, 
and let $\omega$ be an inner function. 
Then $f + \omega H^\infty(\T)$ contains an inner function.
\end{cor}

\begin{proof}
Consider $g:= f/\omega$. 
Then $g \in L^\infty(\T)$ and
\[
\dist(  g,H^\infty(\T)  ) \leq \|g\|_\infty =\|f\|_\infty < 1.
\]
Thus, by Theorem~\ref{T:adamarovkrien}, 
$g+H^\infty(\T)$ contains a unimodular function. 
Therefore, upon multiplying by the inner function $\omega$, 
the set $f + \omega H^\infty(\T)$ 
also contains a unimodular function. 
But $f + \omega H^\infty(\T)\subset H^\infty(\T)$, 
and thus any unimodular function in this set has to be inner.\qed
\end{proof}


\section{Approximation on $\T$ by quotients of inner functions}\label{S:DR}

\begin{center}
\fbox{Goal: $\overline{\mathbf{BP}/\mathbf{BP}} =  \overline{\mathbf{I}/\mathbf{I}} =  \mathbf{U}_{L^\infty(\T)}$}
\end{center}

If $\omega_1$ and $\omega_2$ are inner functions,    
then the quotient $\omega_1/\omega_2$ is unimodular on $\T$. 
But how much  of the family of all unimodular functions on $\T$ do these quotients occupy? 
The Douglas--Rudin theorem provides a satisfactory answer. 
To study this result, we need to examine closely  some special conformal mappings.

Fix the parameter $k$, where $0<k<1$. Let
\begin{eqnarray}
K &:=& \int_{0}^{1} \frac{dt}{\sqrt{(1-t^2)   (1-k^2   t^2)}}, \label{E:defofK}\\
 K' &:=& \int_{0}^{1} \frac{dt}{\sqrt{(1-t^2)   (1-k'^2   t^2)}},
\label{E:defofK'}
\end{eqnarray}
where $k' := \sqrt{1-k^2}$.

Let $\Omega := \C \setminus (-\infty,-1] \cup [1,\infty)$. 
Then the \emph{Jacobi elliptic function} $\sn(z)$, 
or more precisely  $\sn(z,k)$, 
is the conformal mapping  shown in Figure~\ref{F:elliptic function}. 

\begin{figure}[ht]
\begin{center}
\scalebox{0.45}{\includegraphics{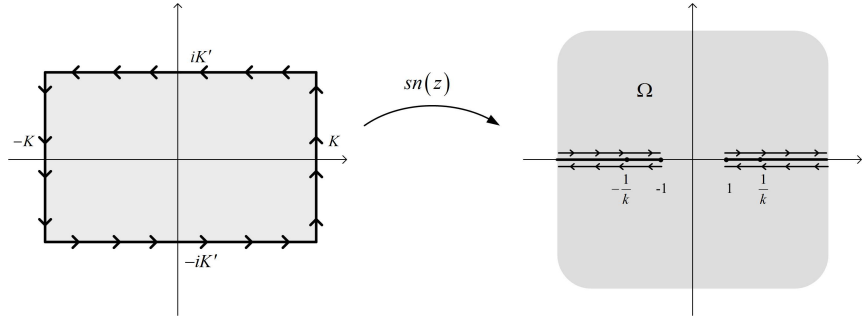}}
\end{center}
\caption{The elliptic function $\sn(z)$}
\label{F:elliptic function}
\end{figure}

The parameter $k$ is free to be any number in the interval $(0,1)$, 
and thus we have a family of elliptic functions. 
The elliptic function $\sn(z)$  \emph{continuously} maps 
the boundaries of the rectangle to the boundary of $\Omega$ in the Riemann sphere, 
i.e. $(-\infty,-1] \cup [1,\infty)\cup \{\infty\}$. 
We emphasize that $\sn$  \emph{continuously} maps the closed rectangle 
$[-K,K] \times [-iK',iK']$ to $\C \cup \{\infty\}$. 
In particular, it maps $\pm iK'$ continuously to $\infty$, 
i.e., if we approach to $\pm iK'$, then $\sn(z)$ tends to infinity. 
However, $\sn$ is not injective on the boundaries of the rectangle. 
If we traverse the path
\[
-iK' \longrightarrow (K-iK') \longrightarrow K \longrightarrow (K+iK') \longrightarrow iK'
\]
on the boundary of the rectangle 
(naively speaking, half of the boundary on the right side), 
then its image under $\sn$ is the interval $[1,\infty]$, 
which is traversed twice in the following manner:
\[
\infty \longrightarrow \frac{1}{k} \longrightarrow 1 \longrightarrow \frac{1}{k} \longrightarrow \infty.
\]
If we continue on the boundary of the rectangle on the path
\[
iK' \longrightarrow (-K+iK') \longrightarrow -K \longrightarrow (-K-iK') \longrightarrow -iK',
\]
then its image under $\sn$ is the interval $[\infty,-1]$, which is traversed twice as
\[
\infty \longrightarrow -\frac{1}{k} \longrightarrow -1 \longrightarrow -\frac{1}{k} \longrightarrow \infty.
\]

Let
\[
\Omega' := \{ z : r < |z| < R \} \setminus (-R,-r),
\]
where
\[
r:=\exp(  -\pi K/K') \hspace{1cm} \mathrm{and} \hspace{1cm}
R:=\exp(  \pi K/K').
\]
Then $g(z):= (K'/\pi)\log z$ maps $\Omega'$ onto the
rectangle $[-K,K] \times [-iK',iK']$. 
Here, $\log$ is the principal branch of the logarithm. 
To better demonstrate the behavior of $g$, 
we put a thin slot on the interval $[-R,-r]$ and study $g$ above and below this slot. 
See Figure~\ref{F:main branch}

\begin{figure}[ht]
\begin{center}
\scalebox{0.5}{\includegraphics{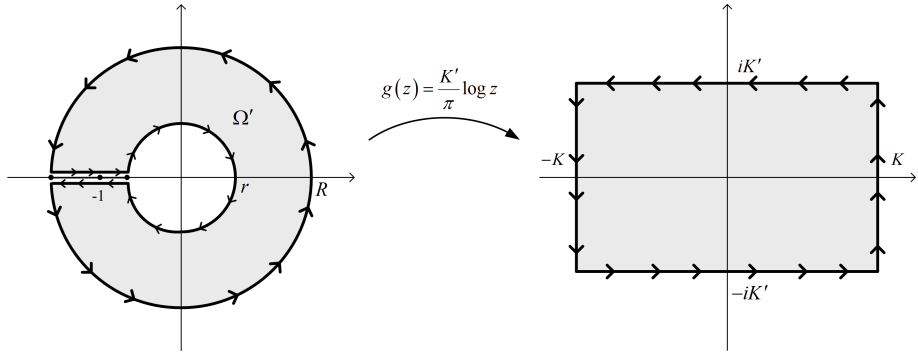}}
\end{center}
\caption{The main branch of logarithm}
\label{F:main branch}
\end{figure}

The conformal mapping $g$ has a continuous extension 
to the closed annulus $\{ z : r \leq |z| \leq R \}$ in the following special manner. 
It is continuous at all points of the circles $\{|z|=r\}$ and $\{|z|=R\}$ 
except at $z=-r$ and $z=-R$. 
If we start from $z=-R$ and traverse counterclockwise the circle $\{|z|=R\}$ 
until we reach this point again, 
then the image of this path under $g$ 
is the segment $\{K\} \times [K-iK',K+iK']$. 
We emphasize that
\[
\lim_{\theta \to -\pi} g(Re^{i\theta}) = K-iK'
\hspace{1cm} \mbox{ and } \hspace{1cm}
\lim_{\theta \to \pi} g(Re^{i\theta}) = K+iK'.
\]
Similarly, if we start from $z=-r$ 
and traverse clockwise the circle $\{|z|=r\}$ 
until we reach this point again, 
then the image of this path under $g$ 
is the segment $\{-K\} \times [-K+iK',-K-iK']$. Note that,
\[
\lim_{\theta \to -\pi} g(re^{i\theta}) = -K-iK'
\hspace{1cm} \mbox{ and } \hspace{1cm}
\lim_{\theta \to \pi} g(re^{i\theta}) = -K+iK'.
\]
Understanding the behavior of $g$ at the points of $[-R,-r]$ is very delicate. 
It depends on the way we approach these points. 
If we approach them from the upper half plane, 
then $g$ continuously and bijectively maps $[-R,-r]$
into the segment $[K,-K] \times \{iK'\}$. 
But, if we approach them from the lower half plane, 
then $g$ continuously and bijectively maps $[-r,-R]$ 
into the segment $[-K,K] \times \{-iK'\}$. 
Therefore, for each $-R \leq x \leq -r$, we have
\[
\lim_{\substack{z \to x\\ \Im z > 0}} g(z) = \frac{K'}{\pi}   \log|x|+iK'
\hspace{0.5cm} \mbox{ and } \hspace{0.5cm}
\lim_{\substack{z \to x\\ \Im z < 0}} g(z) = \frac{K'}{\pi}   \log|x|-iK'.
\]
In particular,
\[
\lim_{\substack{z \to -1\\ \Im z > 0}} g(z) = iK'
\hspace{1cm} \mbox{ and } \hspace{1cm}
\lim_{\substack{z \to -1\\ \Im z < 0}} g(z) = -iK'.
\]

At this point, we combine the last two mappings by defining
\[
h := \sn \circ g.
\]
At first glance, $h$ is a conformal mapping from $\Omega'$ onto $\Omega$. 
But $h$ maps continuously and bijectively $(-R,-1)$ onto $(1/k,\infty)$, 
and $(-1,-r)$ onto $(-\infty,-1/k)$, 
and it also maps continuously  $\{-1\}$ to $\infty$. 
Therefore, by Riemann's theorem, 
$h$ is indeed conformal at all points of $(-R,-r)$ with a simple pole at $\{-1\}$. 
Thus $h$ is a conformal mapping form the annulus $\{ z : r < |z| < R \}$ 
onto $\C \cup \{\infty\} \setminus [-1/k,-1] \cup [1,1/k]$. 
See Figure~\ref{F:confmapping h}.

\begin{figure}[ht]
\begin{center}
\scalebox{0.45}{\includegraphics{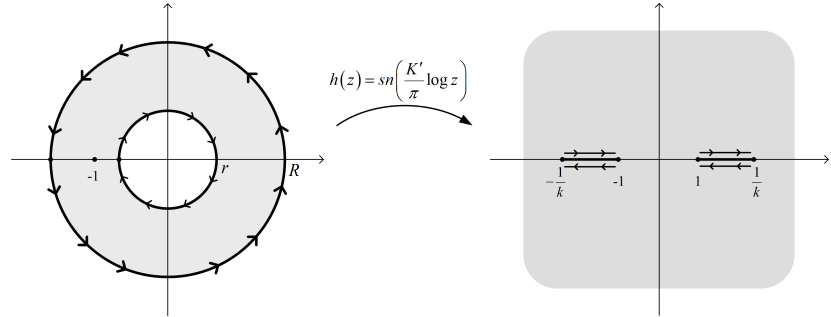}}
\caption{ The conformal mapping $h$}
\label{F:confmapping h}
\end{center}
\end{figure}

We are now ready to define our main conformal mapping. 
Fix $0 < \theta_0 < \pi$, and fix $\varepsilon$ with
\[
0 < \varepsilon < \min\{ \theta_0, ~\pi-\theta_0 \}.
\]
Pick $k \in (0,1)$ such that
\[
\frac{(k-1)^2}{4k} = \frac{\tan(  \frac{\theta_0+\varepsilon}{2}  )}{\tan(\frac{\theta_0}{2})} - 1.
\]
Set 
\begin{equation}  \label{E:defofell'}
\ell := \tan(\theta_0/2)
\qquad\text{and}\qquad
\ell' := \tan(  (\theta_0+\varepsilon)/2  ) = \ell    \bigg(1 + \frac{(k-1)^2}{4k}   \bigg).
\end{equation}
Then the M\"obius transformation
\[
z \longmapsto \frac{k(i-\alpha)   z + (i\beta-\alpha)}{k(i+\alpha)   z +(i\beta+\alpha)},
\]
where
\[
\alpha = \frac{(1+k)   \ell   \tan(\varepsilon/2)}{\ell (1-k) +
2 \tan( \varepsilon/2 ) } \hspace{1cm} \mathrm{and} \hspace{1cm}
\beta = \frac{-(1-k)   \ell + 2k   \tan(\varepsilon/2)}{\ell(1-k) + 2 \tan( \varepsilon/2 ) },
\]
maps the real line into the unit circle in such a way that
\[
-1/k \mapsto 1, \hspace{0.5cm} -1 \mapsto e^{-i\varepsilon},
\hspace{0.5cm} 1 \mapsto  e^{i(\theta_0+\varepsilon)},
\hspace{0.5cm} 1/k \mapsto e^{i\theta_0}.
\]
Moreover,
\[
\infty \mapsto \frac{k(i-\alpha)}{k(i+\alpha)}, \hspace{1cm}
\frac{-(i\beta+\alpha)}{k(i+\alpha)} \mapsto \infty.
\]
Therefore
\[
\Phi(z) := \frac{k(i-\alpha)   h(z) + (i\beta-\alpha)}{k(i+\alpha)
  h(z) + (i\beta+\alpha)}
\]
is a conformal mapping from the annulus $\{ z : r < |z| < R \}$ to
$\C \cup \{\infty\} \setminus \Gamma$, 
where $\Gamma$ consists of two arcs of the unit circle:
\[
\Gamma = \{  e^{i\theta}  :  -\varepsilon \leq \theta \leq 0
 \} \cup \{  e^{i\theta}  :  \theta_0 \leq \theta \leq
\theta_0+\varepsilon  \}.
\]
Figure~\ref{F:confmapping Phi} describes how the boundaries of the annulus are mapped. 

\begin{figure}[ht]
\begin{center}
\scalebox{0.5}{\includegraphics{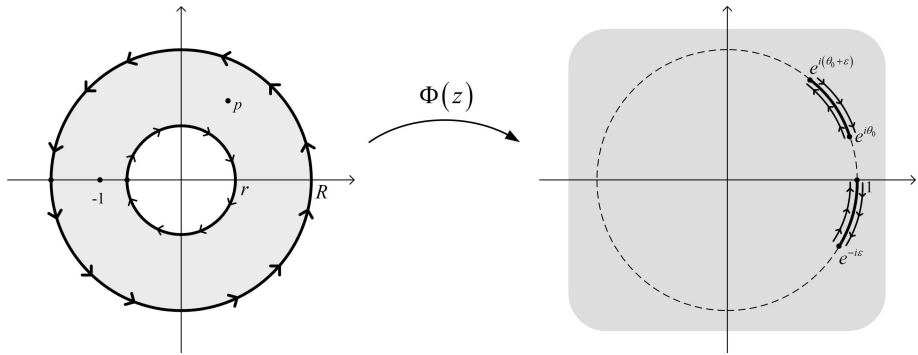}}
\caption{The conformal mapping $\Phi$}
\label{F:confmapping Phi}
\end{center}
\end{figure}

Note that $\Phi$ is conformal at $-1$ with
\[
\Phi(-1) = \frac{i-\alpha}{i+\alpha} \in \T,
\]
and there is a unique point in the annulus, $p$ say, such that
\[
h(p) = - \frac{i\beta+\alpha}{k(i+\alpha)},
\]
and thus $\Phi(p) = \infty$. This point is a simple pole of $\Phi$. Since
$p$ is a simple pole and since $\Phi$ is a conformal mapping, it follows that
$(z-p)   \Phi(z)$ is a \emph{bounded} holomorphic function on the annulus, i.e.,
\begin{equation}  \label{E:z-piphisbdd}
|  (z-p)  \Phi(z)  | \leq C <\infty
\end{equation}
for all $z$ in the annulus.

The conformal mapping $\Phi$ plays a crucial rule in the proof of
the following result of Douglas and Rudin.

\begin{thm}[Douglas--Rudin \cite{MR0254606}]\label{T:dougrud}
Let $\phi \in \mathbf{U}_{L^\infty(\T)}$, 
i.e., a measurable unimodular function on $\T$, 
and let $\varepsilon>0$. 
Then there are inner functions $\omega_1$ and $\omega_2$ (even Blaschke products)
such that
\[
\left\| \phi - \frac{\omega_1}{\omega_2} \right\|_{L^\infty(\T)} < \varepsilon.
\]
\end{thm}

\begin{proof}
First we consider a special class of unimodular functions. 
Let $E$ be a measurable subset of $\T$, 
and let $0<\theta_0<\pi$. Set
\[
\phi( e^{i\theta} ) := \left\{
\begin{array}{ccl}
e^{i\theta_0}, & \mathrm{if} & e^{i\theta} \in E, \\
1, & \mathrm{if} & e^{i\theta} \in \T \setminus E.
\end{array}
\right.
\]
Thus $\phi$ is a unimodular function 
that takes only two different values on $\T$. 
Given $\varepsilon>0$, pick $k$ such that
(\ref{E:defofell'}) holds. Then $K$ and $K'$ are defined respectively by
(\ref{E:defofK}) and (\ref{E:defofK'}). Set
\[
u( e^{i\theta} ) := \left\{
\begin{array}{ccl}
\pi K/K', & \mathrm{if} & e^{i\theta} \in E, \\
-\pi K/K', & \mathrm{if} & e^{i\theta} \in \T \setminus E,
\end{array}
\right.
\]
and let $U=P_r*u$ be its harmonic extension to the open unit disc with
the harmonic conjugate $V=Q_r*u$. Since $-\pi K/K' < U < \pi K/K'$,
the holomorphic function $F=\exp(U+iV)$ maps the unit disc into the
annulus
\[
\{  z  :  \exp(-\pi K/K') < |z| < \exp(\pi K/K')  \}.
\]
Moreover,
\begin{equation}  \label{E:FoncR}
F(e^{i\theta}) = \lim_{r \to 1} F(re^{i\theta}) \in \{  z  :|z|=\exp(\pi K/K')  \}
\end{equation}
for almost all $e^{i\theta} \in E$, and
\begin{equation}  \label{E:Foncr}
F(e^{i\theta}) = \lim_{r \to 1} F(re^{i\theta}) \in \{  z  :|z|=\exp(-\pi K/K')  \}
\end{equation}
for almost all $e^{i\theta} \in \T \setminus E$.

Let $\Psi :=\Phi \circ F$, 
where $\Phi$ is the conformal mapping depicted in Figure~\ref{F:confmapping Phi}. 
Then $\Psi$ is a meromorphic function with poles at the points 
$\{z \in \D  :  F(z) = p \}$. 
Since $\Phi$ has a simple pole at $p$, 
the order of $\Psi$ at a pole $z_0$ is \emph{equal}
the order of $z_0$ as a zero of $F(z)-p$. 
Moreover, since $F(z)-p \in H^\infty(\D)$, 
the zeros of $F(z)-p$ form a Blaschke sequence in $\D$, 
and, by the canonical factorization theorem, 
$F(z)-p$ can be decomposed as
\begin{equation}  \label{D:decompf-p}
F(z)-p = B(z)   S(z)   O(z),
\end{equation}
where $B$ is a Blaschke product, 
$S$ is a singular inner function and $O$ is an outer function.
We shall show that $\omega(z):= B(z)   S(z)   \Psi(z)$ is an inner function 
(note that the product is inner, not $\Psi(z)$ alone).

First of all, since the poles of $\Psi$ are canceled by the zeros of $B$, 
the function $\omega$ is holomorphic on $\D$. Secondly,
\[
|F(e^{i\theta})-p| = |B(e^{i\theta})   S(e^{i\theta}) O(e^{i\theta})| = |O(e^{i\theta})|
\]
for almost all $e^{i\theta} \in \T$. 
Moreover, by (\ref{E:FoncR}),
\[
|O(e^{i\theta})| = |F(e^{i\theta})-p| \geq \exp(\pi K/K') - |p| >0
\]
for almost all $e^{i\theta} \in E$, and , by (\ref{E:Foncr}),
\[
|O(e^{i\theta})| = |F(e^{i\theta})-p| \geq |p| - \exp(-\pi K/K')> 0
\]
for almost all $e^{i\theta} \in \T \setminus E$. 
Thus $|O|$ is bounded away from zero on $\T$, 
which, by Smirnov's theorem, implies that
\[
\frac{1}{O} \in H^\infty(\D).
\]
Finally, by (\ref{E:z-piphisbdd}) and (\ref{D:decompf-p}),
\begin{eqnarray*}
|\omega(z)| &=& |B(z)|    |S(z)|    |\Psi(z)| \\
&=& |B(z)|
|S(z)|  |\Phi( F(z) )| \\
&\leq& |B(z)|    |S(z)|    \frac{C}{|F(z)-p|}\\
&\leq&
\frac{C}{|O(z)|} \leq C'
\end{eqnarray*}
for all $z \in \D$. 
Thus $\omega \in H^\infty(\D)$. 
Moreover, for almost all $e^{i\theta} \in \T$,
\[
\omega(e^{i\theta}) = B(e^{i\theta})   S(e^{i\theta}) \Psi(e^{i\theta}) \in \T.
\]
Therefore, $\omega$ is indeed an inner function.

Turning back to $\Psi$, we note that
\[
\Psi = \frac{\omega}{B    S}
\]
is the quotient of two inner functions.
Also, by
(\ref{E:FoncR}) and the behavior of $\Phi$ on the circle $\{  z:  |z|=\exp(\pi K/K')  \}$, we have
\[
|\phi(e^{i\theta}) - \Psi(e^{i\theta})|  =
|e^{i\theta_0}-\Phi(  F(e^{i\theta}) )|\leq \varepsilon
\]
for almost all $e^{i\theta} \in E$, and, by (\ref{E:Foncr}) and the
behavior of $\Phi$ on the circle $\{  z  :  |z|=\exp(-\pi K/K')\}$, we also have
\[
|\phi(e^{i\theta}) - \Psi(e^{i\theta})| = |1-\Phi(
F(e^{i\theta}) )| \leq \varepsilon
\]
for almost all $e^{i\theta} \in \T \setminus E$. 
This means that
\[
\|\phi - \Psi\|_\infty \leq \varepsilon.
\]

To show that an arbitrary measurable unimodular function 
can be uniformly approximated by the quotient of inner functions, 
we use a simple approximation technique.
Let $\phi$ be a measurable unimodular function. 
Given $\varepsilon>0$, choose $N\geq 1$ such that $2\pi/N < \varepsilon$. Let
\[
E_k := \{  e^{i\theta}  :  2\pi(k-1)/N   \leq \arg
\phi(e^{i\theta}) < 2\pi k/N  \}, \hspace{1cm} (1 \leq k \leq N),
\]
and let
\[
\phi_k( e^{i\theta} ) := \left\{
\begin{array}{ccc}
e^{i2\pi k/N}, & \mathrm{if} & e^{i\theta} \in E_k, \\
1, & \mathrm{if} & e^{i\theta} \ne E_k.
\end{array}
\right.
\]
Then each $\phi_k$ is unimodular 
and takes \emph{only} two different values on $\T$, and
\[
\| \phi - \phi_1   \phi_2   \cdots \phi_N
 \|_\infty \leq \varepsilon.
\]
According to the first part of the proof, there are inner
functions $\omega_{k1}$ and $\omega_{k2}$ such that
\[
\|\phi_k - \omega_{k1}/\omega_{k2} \|_\infty < \varepsilon/N, \hspace{1cm} (1 \leq k \leq N).
\]
Since
\begin{eqnarray*}
\phi - \frac{\omega_{11}}{\omega_{12}}   \frac{\omega_{21}}{\omega_{22}}
  \cdots \frac{\omega_{N1}}{\omega_{N2}}   &=& \phi - \phi_1
\phi_2   \phi_3  \cdots \phi_N \\
&+& (  \phi_1 - \frac{\omega_{11}}{\omega_{12}}  )  \phi_2   \phi_3   \cdots \phi_N \\
&+& \frac{\omega_{11}}{\omega_{12}}   (  \phi_2 - \frac{\omega_{21}}{\omega_{22}}  )
\phi_3  \cdots \phi_N \\
&\vdots& \\&+& \frac{\omega_{11}}{\omega_{12}}   \frac{\omega_{21}}{\omega_{22}}   \cdots (  \phi_N - \frac{\omega_{N1}}{\omega_{N2}}  ),
\end{eqnarray*}
we thus have
\[
\bigg\|  \phi - \frac{\omega_{11}   \omega_{21}   \cdots \omega_{N1}}{\omega_{12}   \omega_{22}   \cdots \omega_{N2}}    \bigg\|_\infty\leq 2\varepsilon.
\]
In the light of Frostman's theorem, 
$\omega_1$ and $\omega_2$ can be replaced by Blaschke products.
This concludes the proof.\qed
\end{proof}


\section{Approximation on $\D$ by convex combinations of quotients of Blaschke products}

\begin{center}
\fbox{Goal: $\overline{\conv(\mathbf{BP}/\mathbf{BP})} =  \overline{\conv(\mathbf{I}/\mathbf{I})} =  \mathbf{B}_{L^\infty(\T)}$}
\end{center}

Clearly, a unimodular measurable function on $\T$
is in the closed unit ball of $L^\infty(\T)$. 
In the first step in studying the closed convex hull of quotients of Blaschke products, 
we show that the family of all unimodular measurable functions on $\T$ 
is a large set in $L^\infty(\T)$, 
in the sense that the closed convex hull of this family  
is precisely the closed unit ball of $L^\infty(\T)$.
The results in this section are taken from \cite{MR0254606}.

\begin{lem} \label{L:cchunimodinl}
Let $f \in \mathbf{B}_{L^\infty(\T)}$ and let $\varepsilon>0$. 
Then there are $u_j \in \mathbf{U}_{L^\infty(\T)}$ 
and  convex weights  $(\lambda_j)_{1 \leq j \leq n}$ such that
\[
\|  \lambda_1   u_1 + \lambda_2   u_2 + \cdots + \lambda_n u_n - f  \|_{L^\infty(\T)} < \varepsilon.
\]
\end{lem}

\begin{proof}
Proceeding precisely as in the proof of Lemma \ref{L:cchunimodinl-count}, we obtain
\[
\bigg|  f(e^{i\theta}) - \frac{1}{N}   \sum_{k=1}^{N}
\frac{e^{i2k\pi/N}+(1-\varepsilon)f(e^{i\theta})}{1+(1-\varepsilon)\overline{f(e^{i\theta})}
  e^{i2k\pi/N}}
  \bigg| \leq
\varepsilon+\frac{4\pi}{\varepsilon N}
\]
for each $e^{i\theta} \in \T$. But each
\[
u_k(e^{i\theta}) =
\frac{e^{i2k\pi/N}+(1-\varepsilon)f(e^{i\theta})}{1+(1-\varepsilon)\overline{f(e^{i\theta})} e^{i2k\pi/N}}
\]
is in fact a unimodular function on $\T$. 
Thus, given $\varepsilon>0$, 
it is enough to choose $N$ so large that $4\pi/(\varepsilon N) <\varepsilon$ to get
\[
\bigg|  f(e^{i\theta}) - \frac{1}{N}   \sum_{k=1}^{N} u_k(e^{i\theta}) \bigg| \leq 2\varepsilon
\]
for all $e^{i\theta} \in \T$.\qed
\end{proof}

Theorem \ref{T:dougrud} and Lemma \ref{L:cchunimodinl} together 
show that the closed convex hull in $L^\infty(\T)$ of the set
\[
\left\{  \frac{\omega_1}{\omega_2}   :   \omega_1 \mbox{ and } \omega_2 \mbox{ are inner}  \right\}
\]
is precisely the closed unit ball of $L^\infty(\T)$. 

\begin{thm}[Douglas--Rudin \cite{MR0254606}] \label{T:cchui/idinl}
Let $f \in \mathbf{B}_{L^\infty(\T)}$ and let $\varepsilon>0$. 
Then there are inner functions $\omega_{ij}$, $1 \leq i,j \leq n$ (even Blaschke products) 
and convex weights $(\lambda_j)_{1 \leq j \leq n}$ such that
\[
\bigg\|  \lambda_1   \frac{\omega_{11}}{\omega_{12}} + \lambda_2
\frac{\omega_{21}}{\omega_{22}} + \cdots + \lambda_n
\frac{\omega_{n1}}{\omega_{n2}} - f  \bigg\|_{L^\infty(\T)} < \varepsilon.
\]
\end{thm}

\begin{proof}
By Lemma \ref{L:cchunimodinl}, 
there are $0 \leq \lambda_1, \lambda_2, \dots , \lambda_n \leq 1$ 
with $\lambda_1+ \cdots + \lambda_n = 1$, 
and  unimodular functions $u_1,   u_2, \dots, u_n$ such that
\[
\|  \lambda_1   u_1 + \lambda_2   u_2 + \cdots + \lambda_n u_n - f  \|_\infty < \varepsilon/2.
\]
Also, for each $k$, by Theorem \ref{T:dougrud}, 
there are inner functions  $\omega_{k1}$ and $\omega_{k2}$ such that
\[
\| u_k - \omega_{k1}/\omega_{k2}  \|_\infty < \varepsilon/2.
\]
Then
\[
\bigg\|  f - \sum_{k=1}^{n} \lambda_k   \omega_{k1}/\omega_{k2}
 \bigg\|_\infty \leq \bigg\|  f - \sum_{k=1}^{n} \lambda_k
u_k  \bigg\|_\infty + \sum_{k=1}^{n} \lambda_k \|   u_k -
\omega_{k1}/\omega_{k2}  \|_\infty < \varepsilon.
\]
This completes the proof.\qed
\end{proof}

\emph{Remark.}
Since the product of two inner functions is an inner function, 
in the quotients appearing in Theorem \ref{T:cchui/idinl}, 
we can take a common denominator and thus, without loss of generality, 
assume that all the $\omega_{k2}$ are equal.
Hence, under the conditions of Theorem \ref{T:cchui/idinl}, 
there are inner functions $\omega$ and $\omega_1,\dots,\omega_n$ such that
\[
\bigg\|  \lambda_1   \frac{\omega_1}{\omega} + \lambda_2
\frac{\omega_2}{\omega} + \cdots + \lambda_n
\frac{\omega_n}{\omega} - f  \bigg\|_\infty < \varepsilon.
\]
The same remark obviously applies to quotients of Blaschke products.


\section{Approximation on $\D$ by convex combination of infinite Blaschke products}

\begin{center}
\fbox{Goal: $\overline{\conv(\mathbf{BP})} =  \overline{\conv(\mathbf{I})} =  \mathbf{B}_{H^\infty}$}
\end{center}

To study convex combinations of Blaschke products, we need the following variant of Theorem \ref{T:cchui/idinl}.

\begin{lem} \label{L:finhinftyisinclosw/w}
Let $f \in H^\infty$ and let $\varepsilon>0$. Then there are
real constants $a_j$ and inner functions
$\omega$ and $\omega_j$ such that
\[
a_1   \frac{\omega_1}{\omega} + a_2   \frac{\omega_2}{\omega} +
\cdots + a_n   \frac{\omega_n}{\omega} \in H^\infty
\]
and
\[
\bigg\|  a_1   \frac{\omega_1}{\omega} + a_2
\frac{\omega_2}{\omega} + \cdots + a_n   \frac{\omega_n}{\omega}
- f  \bigg\|_\infty < \varepsilon.
\]
\end{lem}

\begin{proof}
The result is clear if $f \equiv 0$, 
so let us assume that $f \not \equiv 0$.
By the remark following Theorem~\ref{T:cchui/idinl}, 
there are $0 \leq \lambda_1, \lambda_2, \dots , \lambda_m \leq 1$ 
with $\lambda_1 + \lambda_2 + \cdots + \lambda_m = 1$, 
and  inner functions $\omega_1,   \omega_2, \dots ,\omega_m$ and $\omega$ 
such that
\begin{equation} \label{E:testimw}
\bigg\|  \lambda_1   \frac{\omega_1}{\omega} + \lambda_2
\frac{\omega_2}{\omega} + \cdots + \lambda_m
\frac{\omega_m}{\omega} - \frac{f}{\|f\|_\infty}  \bigg\|_\infty < \varepsilon',
\end{equation}
where $\varepsilon'=\varepsilon/(2 \|f\|_\infty)$. Put
\[
F := \frac{1}{\varepsilon'}    \bigg(  \lambda_1
\frac{\omega_1}{\omega} + \lambda_2   \frac{\omega_2}{\omega} +
\cdots + \lambda_m   \frac{\omega_m}{\omega}  \bigg).
\]
Then $F \in L^\infty(\T)$, and the last inequality shows that
\[
\dist(  F,H^\infty(\T)  ) < 1.
\]
Hence, by Theorem \ref{T:adamarovkrien}, 
there are $G \in H^\infty(\T)$ and a unimodular function $I$  such that $F=I+G$. 
But, since $\omega F$ is in $H^\infty(\T)$, 
the function $\omega_0:=\omega I = \omega F - \omega G$ 
is a unimodular function in $H^\infty(\T)$. 
In other words, $\omega_0$ is an inner function, 
and thus $I=\omega_0/\omega$ is the quotient of two inner functions. Moreover,
\[
\lambda_1   \frac{\omega_1}{\omega} + \lambda_2
\frac{\omega_2}{\omega} + \cdots + \lambda_m
\frac{\omega_m}{\omega} - \varepsilon'   \frac{\omega_0}{\omega} =
\varepsilon'(F-I) = \varepsilon' G \in H^\infty(\T),
\]
and, by (\ref{E:testimw}),
\[
\bigg\|  \lambda_1   \frac{\omega_1}{\omega} + \lambda_2
\frac{\omega_2}{\omega} + \cdots + \lambda_m
\frac{\omega_m}{\omega} - \varepsilon'   \frac{\omega_0}{\omega} -
\frac{f}{\|f\|_\infty}  \bigg\|_\infty < 2\varepsilon'.
\]
Therefore,
\[
\bigg\|   a_1   \frac{\omega_1}{\omega} + a_2
\frac{\omega_2}{\omega} + \cdots + a_m   \frac{\omega_m}{\omega}
+ a_{m+1}   \frac{\omega_{m+1}}{\omega} - f  \bigg\|_\infty < 2
\|f\|_\infty  \varepsilon'=\varepsilon,
\]
where $a_k:= \lambda_k   \|f\|_\infty$, for $1 \leq k \leq m$,
and $a_{m+1}:=-\varepsilon' \|f\|_\infty$ and 
$\omega_{m+1} :=\omega_0$.\qed
\end{proof}

Now we are able to show that  the closed convex hull of
the family of all inner functions on $\T$ 
is precisely the closed unit ball of $H^\infty(\T)$.

\begin{thm}[Marshall \cite{MR0402054}] \label{T:cchuinnerinl}
Let $f \in \mathbf{B}_{H^\infty}$ and let $\varepsilon>0$. 
Then there are  inner functions $\omega_j$ (even Blaschke products) 
and  convex weights $(\lambda_j)_{1 \leq j \leq n}$ such that
\[
\|  \lambda_1   \omega_1 + \lambda_2   \omega_2 + \cdots +
\lambda_n   \omega_n - f  \|_\infty < \varepsilon.
\]
\end{thm}

\begin{proof}
By Lemma \ref{L:finhinftyisinclosw/w}, 
there are real constants $a_1,   \dots,   a_n$ 
and inner functions $\omega, \omega_1,   \dots, \omega_n$ such that
\[
g := a_1   \frac{\omega_1}{\omega} + a_2
\frac{\omega_2}{\omega} + \cdots + a_n   \frac{\omega_n}{\omega}
\in H^\infty(\T),
\]
and $\|g - (1-2\varepsilon)f  \|_\infty < \varepsilon$. 
Hence it is enough to approximate $g$ by convex combination of inner functions. 
Note that $\|g\|_\infty < 1-\varepsilon$.

Set
\[
\omega_0:=\omega_1   \omega_2   \cdots   \omega_n.
\]
Since
\[
\overline{g}= \overline{a}_1
\frac{\overline{\omega}_1}{\overline{\omega}} + \overline{a}_2
\frac{\overline{\omega}_2}{\overline{\omega}} + \cdots + \overline{a}_n
  \frac{\overline{\omega}_n}{\overline{\omega}} =  a_1
\frac{\omega}{\omega_1} + a_2   \frac{\omega}{\omega_2} + \cdots
+ a_n   \frac{\omega}{\omega_n},
\]
we clearly have
\[
\omega_0   \overline{g} \in H^\infty(\T).
\]
This property is the main advantage of $g$ over $f$. 
Now we follow a similar procedure to that in the proof of Lemma~\ref{L:cchunimodinl-count}.

Let $w \in \D$ and $\gamma \in \T$. 
Then, by the Cauchy integral formula,
\[
w = \frac{1}{2\pi i} \int_{\T} \frac{\gamma
\zeta+w}{\zeta(1+\overline{w}   \gamma   \zeta)}  \, d\zeta =
\frac{1}{2\pi} \int_{0}^{2\pi} \frac{\gamma
e^{i\theta}+w}{1+\overline{w}   \gamma   e^{i\theta}}  \,d\theta.
\]
Since
\[
\bigg|  \frac{\gamma   e^{i\theta}+w}{1+\overline{w}   \gamma
  e^{i\theta}} - \frac{\gamma   e^{i\theta'}+w}{1+\overline{w}
  \gamma   e^{i\theta'}}  \bigg| \leq \frac{1+|w|}{1-|w|}|\theta-\theta'|,
\]
for
\[
\Delta_N: = \bigg|  w - \frac{1}{N}   \sum_{k=1}^{N} \frac{\gamma
e^{i2k\pi/N}+w}{1+\overline{w}   \gamma   e^{i2k\pi/N}} \bigg|,
\]
we have the estimation
\begin{eqnarray*}
\Delta_N &=& \bigg|  \frac{1}{2\pi}   \int_{0}^{2\pi} \frac{\gamma
e^{i\theta}+w}{1+\overline{w}   \gamma   e^{i\theta}}  \,
d\theta - \frac{1}{N}   \sum_{k=1}^{N} \frac{\gamma
e^{i2k\pi/N}+w}{1+\overline{w} \gamma   e^{i2k\pi/N}} \bigg| \\
&\leq& \frac{1}{2\pi}   \sum_{k=1}^{N}
\int_{2(k-1)\pi/N}^{2k\pi/N} \bigg|  \frac{\gamma
e^{i\theta}+w}{1+\overline{w}   \gamma   e^{i\theta}} -
\frac{\gamma   e^{i2k\pi/N}+w}{1+\overline{w} \gamma
e^{i2k\pi/N}} \bigg|  \, d\theta\\
&\leq& \frac{1}{2\pi}   \sum_{k=1}^{N}
\int_{2(k-1)\pi/N}^{2k\pi/N} \frac{1+|w|}{1-|w|} \frac{2\pi}{N}  \, d\theta \\
&=& \frac{1+|w|}{1-|w|} \frac{2\pi}{N}.
\end{eqnarray*}
Hence, for almost all $e^{i\theta} \in \T$, 
setting $w:=g(e^{i\theta})$ and $\gamma:=\omega_0(e^{i\theta})$, 
we get
\[
\bigg|  g(e^{i\theta}) - \frac{1}{N}   \sum_{k=1}^{N}
\frac{\omega_0(e^{i\theta})
e^{i2k\pi/N}+g(e^{i\theta})}{1+\overline{g(e^{i\theta})}
\omega_0(e^{i\theta})   e^{i2k\pi/N}}
  \bigg| \leq \frac{1+|g(e^{i\theta})|}{1-|g(e^{i\theta})|}
\frac{2\pi}{N}.
\]
Thus, for almost all $e^{i\theta} \in \T$,
\[
\bigg|  f(e^{i\theta}) - \frac{1}{N}   \sum_{k=1}^{N}
\frac{\omega_0(e^{i\theta})
e^{i2k\pi/N}+g(e^{i\theta})}{1+\overline{g(e^{i\theta})}
\omega_0(e^{i\theta})   e^{i2k\pi/N}}
  \bigg| \leq
3\varepsilon+\frac{4\pi}{\varepsilon N}.
\]
But, for each $k$,
\[
\omega_k(e^{i\theta}) := \frac{\omega_0(e^{i\theta})
e^{i2k\pi/N}+g(e^{i\theta})}{1+\overline{g(e^{i\theta})}
\omega_0(e^{i\theta})   e^{i2k\pi/N}}
\]
is in fact an inner function, 
since in the first place it is a unimodular function, 
and besides $g,   \omega_0,   \overline{g}\omega_0 \in H^\infty(\T)$ 
and $|1+\overline{g(e^{i\theta})} \omega_0(e^{i\theta})   e^{i2k\pi/N}| \geq \varepsilon$, 
for almost all $e^{i\theta} \in \T$. 
Therefore, given $\varepsilon>0$, 
it is enough to choose $N$ so large that $4\pi/(\varepsilon N) < \varepsilon$ to get
\[
\bigg|  f(e^{i\theta}) - \frac{1}{N}   \sum_{k=1}^{N} \omega_k(e^{i\theta} \bigg| \leq 4\varepsilon
\]
for almost all $e^{i\theta} \in \T$. 

By Frostman's theorem, there are Blaschke products $B_1,\dots,B_n$ such that
$\| \omega_k - B_k \|_\infty < \varepsilon/2$, for each $1 \leq k \leq n$. Hence
\[
\bigg\|  f - \sum_{k=1}^{n} \lambda_k   B_k  \bigg\|_\infty
\leq \bigg\|  f - \sum_{k=1}^{n} \lambda_k   \omega_k
 \bigg\|_\infty + \sum_{k=1}^{n} \lambda_k   \|  B_k - \omega_k
 \|_\infty < \varepsilon.
\]
This completes the proof.\qed
\end{proof}


\section{An application: the Halmos conjecture}

Let $H$ be a complex Hilbert space
and  $T$ be a bounded linear operator on~$H$.
The \emph{numerical range} of $T$ is defined by
\[
W(T):=\{\langle Tx,x\rangle: x\in H,~\|x\|=1\}.
\]
It is a convex set
whose closure contains the spectrum of $T$.
If $\dim H<\infty$, then $W(T)$ is compact.
The \emph{numerical radius} of $T$ is defined by
\[
w(T):=\sup\{|\langle Tx,x\rangle|: x\in H,~\|x\|=1\}.
\]
It is related to the operator norm via
the   double inequality
\begin{equation}\label{E:nrnorm}
\|T\|/2\le w(T) \le \|T\|.
\end{equation}
If further $T$ is self-adjoint, then $w(T)=\|T\|$. In contrast with spectra,
it is not true in general that $W(p(T))=p(W(T))$ for polynomials $p$,
nor is it true if we take convex hulls of both sides.
However, some partial results do hold.
Perhaps the most famous of these is the power inequality:
for all $n\ge1$, we have
\[
w(T^n)\le w(T)^n.
\]
This was conjectured by Halmos and, after  several  partial results,
was established by Berger using dilation theory.
An elementary proof was given by Pearcy in \cite{MR0201976}.
A more general result was established by Berger and Stampfli in \cite{MR0222694}.
They showed that, if $w(T)\le1$, then, for all $f$ in the disk algebra with $f(0)=0$, we have
\[
w(f(T))\le\|f\|_\infty.
\]
Again their proof  used dilation theory. 
We give an elementary proof of this result
along the lines of Pearcy's proof of the power inequality. 

We require two folklore lemmas about finite Blaschke products.

\begin{lem}\label{L:Blaschke1}
Let $B$ be a finite Blaschke product.
Then $\zeta B'(\zeta)/B(\zeta)$ is real and strictly positive for all $\zeta\in\T$.
\end{lem}

\begin{proof}
We can write
\[
B(z)=c\prod_{k=1}^n\frac{a_k-z}{1-\overline{a}_kz},
\]
where $a_1,\dots,a_n\in\D$ and $c\in\T$. Then
\[
\frac{B'(z)}{B(z)}
=\sum_{k=1}^n\frac{1-|a_k|^2}{(z-a_k)(1-\overline{a}_kz)}.
\]
In particular, if $\zeta\in\T$, then
\[
\frac{\zeta B'(\zeta)}{B(\zeta)}
=\sum_{k=1}^n\frac{1-|a_k|^2}{|\zeta-a_k|^2},
\]
which is real and strictly positive.\qed
\end{proof}

\begin{lem}\label{L:Blaschke2}
Let $B$ be a Blaschke product of degree $n$ such that $B(0)=0$.
Then, given $\gamma\in\T$,
there exist $\zeta_1,\dots,\zeta_n\in\T$
and $c_1,\dots,c_n>0$ such that
\begin{equation}\label{E:Blaschke2}
\frac{1}{1-\overline{\gamma}B(z)}
=\sum_{k=1}^n\frac{c_k}{1-\overline{\zeta}_kz}.
\end{equation}
\end{lem}

\begin{proof}
Given $\gamma\in\T$,
the roots of the equation $B(z)=\gamma$ lie on the unit circle,
and by Lemma~\ref{L:Blaschke1} they are simple.
Call them $\zeta_1,\dots,\zeta_n$.
Then $1/(1-\overline{\gamma}B)$ has simple poles at the $\zeta_k$.
Also, as $B(0)=0$, we have $B(\infty)=\infty$
and so $1/(1-\overline{\gamma}B)$ vanishes at $\infty$.
Expanding it in partial fractions gives \eqref{E:Blaschke2},
for some choice of $c_1,\dots,c_n\in\C$.

The coefficients $c_k$ are easily evaluated.
Indeed, from \eqref{E:Blaschke2} we have
\[
c_k
=\lim_{z\to\zeta_k}\frac{1-\overline{\zeta}_kz}{1-\overline{\gamma}B(z)}
=\lim_{z\to\zeta_k}\frac{(\zeta_k-z)/\zeta_k}{(B(\zeta_k)-B(z))/B(\zeta_k)}
=\frac{B(\zeta_k)}{\zeta_kB'(\zeta_k)}.
\]
In particular $c_k>0$ by Lemma~\ref{L:Blaschke1}.\qed
\end{proof}

\begin{thm}[Berger--Stampfli \cite{MR0222694}]\label{T:Berger}
Let $H$ be a complex Hilbert space,
let $T$ be a bounded linear operator on $H$ with $w(T)\le1$,
and let $f$ be a function in the disk algebra such that $f(0)=0$.
Then $w(f(T))\le\|f\|_\infty$.
\end{thm}

\begin{proof}
(Klaja--Mashreghi--Ransford \cite{MR3487232}.) 
Suppose first that $f$ is a finite Blaschke product $B$.
Suppose also that the spectrum $\sigma(T)$ of $T$
lies within the open unit disk $\D$.
By the spectral mapping theorem
$\sigma(B(T))=B(\sigma(T))\subset\D$ as well.
Let $x\in H$ with $\|x\|=1$.
Given $\gamma\in\T$,
let $\zeta_1,\dots,\zeta_n\in\T$ and $c_1,\dots,c_n>0$
as in Lemma~\ref{L:Blaschke2}. Then we have
\begin{align*}
1-\overline{\gamma}\langle B(T)x,x\rangle
&=\langle (I-\overline{\gamma}B(T))x,x\rangle\\
&=\langle y,(I-\overline{\gamma}B(T))^{-1}y\rangle
&&\text{where~}y:=(I-\overline{\gamma}B(T))x\\
&=\Bigl\langle y,\sum_{k=1}^nc_k (I-\overline{\zeta}_kT)^{-1}y\Bigr\rangle
&&\text{by~}\eqref{E:Blaschke2}\\
&=\sum_{k=1}^n c_k\langle(I-\overline{\zeta}_kT)z_k,z_k\rangle
&&\text{where~}z_k:=(I-\overline{\zeta}_kT)^{-1}y\\
&=\sum_{k=1}^n c_k(\|z_k\|^2-\overline{\zeta}_k\langle Tz_k,z_k\rangle).
\end{align*}Since $w(T)\le1$, we have
$\Re(\|z_k\|^2-\overline{\zeta}_k\langle Tz_k,z_k\rangle)\ge0$,
and as $c_k>0$ for all $k$, it follows that
\[
\Re(1-\overline{\gamma}\langle B(T)x,x\rangle)\ge0.
\]
As this holds for all $\gamma\in\T$ and all $x$ of norm $1$,
it follows that $w(B(T))\le1$.

Next we relax the assumption on $f$,
still assuming that $\sigma(T)\subset\D$.
We can suppose that $\|f\|_\infty=1$.
Then, by Carath\'eodory's theorem (Theorem~\ref{T:caratheodory-fbp}),
there exists a sequence of finite Blaschke products $B_n$
that converges locally uniformly to $f$ in $\D$.
Moreover, as $f(0)=0$, we can also arrange that $B_n(0)=0$ for all $n$.
By what we have proved, $w(B_n(T))\le1$ for all $n$.
Also $B_n(T)$ converges in norm to $f(T)$, because $\sigma(T)\subset\D$.
It follows that $w(f(T))\le1$, as required.

Finally we relax the assumption that $\sigma(T)\subset\D$.
By what we have already proved,
$w(f(rT))\le\|f\|_\infty$ for all $r<1$.
Interpreting $f(T)$ as $\lim_{r\to1^-}f(rT)$,
it follows that  $w(f(T))\le\|f\|_\infty$,
\emph{provided} that this limit exists.
In particular this is true when $f$ is holomorphic in a neighborhood of $\overline{\D}$.
To prove the existence of the limit in the general case,
we proceed as follows.
Given $r,s\in(0,1)$,
the function $g_{rs}(z):=f(rz)-f(sz)$
is holomorphic in a neighborhood of $\overline{\D}$
and vanishes at~$0$,
so, by what we have already proved,
$w(g_{rs}(T))\le\|g_{rs}\|_\infty$. Therefore,
\[
\|f(rT)-f(sT)\|=\|g_{rs}(T)\|\le 2w(g_{rs}(T))\le 2\|g_{rs}\|_\infty.
\]
The right-hand side tends to zero as $r,s\to1^{-}$, so,
by the usual Cauchy-sequence argument,
$f(rT)$ converges as $r\to1^-$.
This completes the proof.\qed
\end{proof}

\emph{Remark.}
The assumption that $f(0)=0$ is essential in the Berger--Stampfli theorem.
Without this assumption, the situation becomes more complicated.
The best result in this setting is Drury's teardrop theorem \cite{MR2401640}. 
See also \cite{MR3487232} for an alternative proof.


\bibliography{biblio}
\bibliographystyle{spmpsci}

\end{document}